\pdfoutput=1
\documentclass[onefignum,onetabnum,reqno]{siamart220329}

\usepackage{stmaryrd}%maybe double braces
\usepackage{float}
\usepackage{microtype}
\usepackage{upgreek}
\usepackage{mathtools}
\usepackage{paper_diening}
\usepackage{mathrsfs}
\usepackage{enumitem}
\usepackage{hyperref}
\usepackage{booktabs}
\usepackage{colortbl}
\usepackage{tabulary}
\usepackage{diagbox}
\usepackage{caption}

\newsiamremark{assumption}{Assumption}
\usepackage{lipsum}
\usepackage{amsfonts}
\usepackage{graphicx}
\usepackage{epstopdf}
\usepackage{algorithmic}
\ifpdf
  \DeclareGraphicsExtensions{.eps,.pdf,.png,.jpg}
\else
  \DeclareGraphicsExtensions{.eps}
\fi

\newsiamremark{remark}{Remark}
\newsiamremark{hypothesis}{Hypothesis}
\crefname{hypothesis}{Hypothesis}{Hypotheses}
\newsiamthm{claim}{Claim}

\headers{LDG approximation for the \lowercase{$p$}-Navier--Stokes system}{A. Kaltenbach, M. R\r{U}\v{Z}I\v{C}KA}

\title{A Local Discontinuous Galerkin approximation for the \lowercase{$p$}-Navier--Stokes system, Part II: Convergence rates for the velocity\thanks{Submitted to the editors \today.
}}

\author{Alex Kaltenbach\thanks{Department of Applied Mathematics, University of Freiburg, Ernst--Zermelo-Straße 1, D-79104 Freiburg, GERMANY. (\email{alex.kaltenbach@mathematik.uni-freiburg.de}).}
\and Michael R\r{U}\v{Z}I\v{C}KA\thanks{Department of Applied Mathematics, University of Freiburg, Ernst--Zermelo-Straße 1, D-79104 Freiburg, GERMANY.
  (\email{rose@mathematik.uni-freiburg.de}).}}

\usepackage{amsopn}

\usepackage{xcolor}
\definecolor{rltred}{rgb}{0.75,0,0}
\definecolor{rltgreen}{rgb}{0,0.5,0}
\definecolor{rltblue}{rgb}{0,0,0.75}
\ifpdf
\hypersetup{
  pdftitle={A Local Discontinuous Galerkin Approximation for the $p$-Navier--Stokes system},
  pdfauthor={A. Kaltenbach, M. \Ruzicka}
}
\fi

\providecommand{\meantmp}[2]{#1\langle{#2}#1\rangle}
\providecommand{\mean}[1]{\meantmp{}{#1}}

\providecommand{\jumptmp}[2]{#1\llbracket{#2}#1\rrbracket}
\providecommand{\jump}[1]{\jumptmp{}{#1}}

\providecommand{\avgtmp}[2]{#1\{{#2}#1\}}
\providecommand{\avg}[1]{\avgtmp{}{#1}}
\providecommand{\bigavg}[1]{\avgtmp{\big}{#1}}

\providecommand{\flux}[1]{{\widehat{#1}}}

\providecommand{\PiDG}{{\Uppi_{h}^{k}}}

\providecommand{\Pia}{{\Uppi_h^{0}}}

\providecommand{\Vo}{{\mathaccent23 V}}
\providecommand{\Qo}{{\mathaccent23 Q}}
\providecommand{\Xhk}{\smash{X_h^k}}
\providecommand{\Vhk}{\smash{V_h^k}}

\providecommand{\Qhkc}{\smash{Q_{h,c}^k}}
\providecommand{\Qhko}{{{\mathaccent23 Q}_h^k}}
\providecommand{\Qhkco}{{{\mathaccent23 Q}_{h,c}^k}}

\providecommand{\SSS}{\boldsymbol{\mathcal{S}}}

\newcommand{\Ghk}{\boldsymbol{\mathcal{G}}_h^k}
\newcommand{\Ghnk}{\boldsymbol{\mathcal{G}}_{h_n}^k\! }
\newcommand{\Dhk}{\boldsymbol{\mathcal{D}}_h^k}

\newcommand{\Divhk}{\mathcal{D}\dot{\iota}\nu_h^k}

\newcommand{\Rhk}{\boldsymbol{\mathcal{R}}_h^k}

\newcommand{\WDG}{W^{1,p}(\mathcal{T}_h)}

\providecommand{\divo}{\textrm{div}\,}

\providecommand{\sss}{\avg{\abs{\Pia \Dhk \bfv_h}}}
\providecommand{\sssl}{\avg{\abs{\Pia \Dhk \bfv_h}}}

\begin{document}

\maketitle

% REQUIRED
\begin{abstract}
  In the present paper, we prove convergence rates for the velocity of the Local
  Discontinuous Galerkin (LDG) approximation, proposed in Part I of the paper (cf.~\cite{kr-pnse-ldg-1}), of
  systems of $p$-Navier--Stokes~type and $p$-Stokes type with $p\in (2,\infty)$. The
  convergence rates are  optimal for linear ansatz functions. The results are supported by
  numerical experiments.
\end{abstract}

% REQUIRED

% REQUIRED

\section{Introduction}

In this paper, we continue our study of the Local Discontinuous Galerkin
(LDG) scheme, proposed in Part I of the paper (cf.~\cite{kr-pnse-ldg-1}), of
steady systems~of $p$-Navier\hspace{-0.2mm}--\hspace{-0.2mm}Stokes \hspace{-0.2mm}type. \hspace{-0.2mm}In \hspace{-0.2mm}this \hspace{-0.2mm}paper, \hspace{-0.2mm}we \hspace{-0.2mm}restrict
\hspace{-0.2mm}ourselves \hspace{-0.2mm}to \hspace{-0.2mm}the \hspace{-0.2mm}homogeneous~\hspace{-0.2mm}problem,~\hspace{-0.2mm}i.e., 
\begin{equation}
  \label{eq:p-navier-stokes}
  \begin{aligned}
    -\divo\SSS(\bfD\bfv)+[\nabla\bfv]\bfv+\nabla q&=\bfg  \qquad&&\text{in }\Omega\,,\\
    \divo\bfv&=0 \qquad&&\text{in }\Omega\,,
    \\
    \bfv &= \mathbf{0} &&\text{on } \partial\Omega\,.
  \end{aligned}
\end{equation}
This system describes the steady motion of a homogeneous,
incompressible~fluid~with shear-dependent viscosity. More precisely,
for a given vector field $\bfg\colon\Omega\to \setR^d$ describing external
body forces and a homogeneous Dirichlet
boundary~condition~\eqref{eq:p-navier-stokes}$_3$, we seek for a
velocity vector field~${\bfv=(v_1,\dots,v_d)^\top\colon \Omega\to
  \setR^d}$~and~a~scalar~kinematic~pressure ${q\colon \Omega\to \setR}$ solving~\eqref{eq:p-navier-stokes}.
Here, $\Omega\hspace{-0.1em}\subseteq\hspace{-0.1em} \mathbb{R}^d$, $d\hspace{-0.1em}\in\hspace{-0.1em} \set{2,3}$, is a bounded polyhedral domain having a Lipschitz continuous boundary $\partial\Omega$. The extra stress tensor $\SSS(\bfD\bfv)\colon\Omega\to \setR^{d\times d}_{\textup{sym}}$~depends on the strain rate tensor $\smash{\bfD\bfv\coloneqq \frac{1}{2}(\nabla\bfv+\nabla\bfv^\top)\colon\Omega\to \setR^{d\times d}_{\textup{sym}}}$, i.e., the symmetric part of the velocity gradient  $\bfL\coloneqq \nabla\bfv%=(\partial_j v_i)_{i,j=1,\dots,d}
\colon\Omega\to \setR^{d\times d}$.~The~\mbox{convective}~term~$\smash{[\nabla\bfv]\bfv\colon\Omega\to \mathbb{R}^d}$ is defined via $\smash{([\nabla\bfv]\bfv)_i\coloneqq \sum_{j=1}^d{v_j\partial_j v_i}}$ for all $i=1,\dots,d$.

Throughout the paper, we  assume that the extra stress tensor~$\SSS$~has~\mbox{$(p,\delta)$-structure} (cf.~Assumption~\ref{assum:extra_stress}). The relevant example falling into this class is 
\begin{align*}
    \smash{\SSS(\bfD\bfv)=\mu\, (\delta+\vert \bfD\bfv\vert)^{p-2}\bfD\bfv}\,,
\end{align*}
where $p\in (1,\infty)$, $\delta\ge 0$, and $\mu>0$.

For a discussion of the model and the state of the art, we refer to
Part~I~of~the~paper (cf.~\cite{kr-pnse-ldg-1}). As already pointed out, to
the best of the authors knowledge, there are no investigations using DG
methods for the $p$-Navier--Stokes problem
\eqref{eq:p-navier-stokes}.~In~this~paper, we continue the
investigations of Part I of the paper  (cf.~\cite{kr-pnse-ldg-1}) and prove~\mbox{convergence}~rates  for the velocity of the $p$-Navier--Stokes problem
\eqref{eq:p-navier-stokes} under the assumption~that~the velocity and
$\bfg$ satisfy natural regularity conditions and a %usual
smallness~condition for the velocity~in~the~energy~norm. We restrict ourselves to the
case $p \in (2,\infty)$. Our approach is inspired by the results in
\cite{dkrt-ldg} and \cite{kr-phi-ldg}. The convergence rates are
optimal for linear ansatz function, due to the new flux in the
stabilization term (cf.~\eqref{def:flux-S}). The same results are obtained for the $p$-Stokes problem without smallness condition.\newpage

%\comment{p-stokes?}

\textit{This paper is organized as follows:} \!In Section
\ref{sec:preliminaries}, we introduce the employed~\mbox{notation}, define
 relevant function spaces, 
 basic assumptions on the extra stress~tensor~$\SSS$~and~its consequences, the weak formulations Problem (Q) and
Problem~(P)~of~the~system~\eqref{eq:p-navier-stokes}, and the  discrete
operators.  In Section \ref{sec:ldg}, we define our numerical
fluxes~and~derive~the flux \hspace{-0.1mm}and \hspace{-0.1mm}the \hspace{-0.1mm}primal
\hspace{-0.1mm}formulation, \hspace{-0.1mm}i.e, \hspace{-0.1mm}Problem~\hspace{-0.1mm}(Q$_h$) \hspace{-0.1mm}and \hspace{-0.1mm}Problem~\hspace{-0.1mm}(P$_h$),~\hspace{-0.1mm}of~\hspace{-0.1mm}the~\hspace{-0.1mm}\mbox{system}~\hspace{-0.1mm}\eqref{eq:p-navier-stokes}.  In Section
\ref{sec:rates}, we~\mbox{derive} error estimates for our problem
(cf.~Theorem~\ref{thm:error}, Corollary~\ref{cor:error}).  These
are the first convergence rates for a DG method for systems of
$p$-Navier--Stokes type. In Section \ref{sec:experiments}, we confirm
our theoretical findings via numerical experiments.

\section{Preliminaries}\label{sec:preliminaries}

\subsection{Function \hspace*{-0.1mm}spaces}

\!\!We \hspace*{-0.1mm}use \hspace*{-0.1mm}the \hspace*{-0.1mm}same \hspace*{-0.1mm}notation \hspace*{-0.1mm}as \hspace*{-0.1mm}in \hspace*{-0.1mm}Part~\hspace*{-0.1mm}I~\hspace*{-0.1mm}of~\hspace*{-0.1mm}the~\hspace*{-0.1mm}paper~\hspace*{-0.1mm}(cf.~\hspace*{-0.1mm}\cite{kr-pnse-ldg-1}). For the convenience of the reader, we repeat some of it.

We employ $c, C>0$ to denote generic constants, that may change from line
to line, but are not depending on the crucial quantities. For $k\in \setN$ and $p\in [1,\infty]$, we employ the customary
Lebesgue spaces $(L^p(\Omega), \smash{\norm{\cdot}_p}) $ and Sobolev
spaces $(W^{k,p}(\Omega), \smash{\norm{\cdot}_{k,p}})$, where $\Omega \hspace{-0.1em}\subseteq \hspace{-0.1em} \setR^d$, $d \hspace{-0.1em}\in \hspace{-0.1em} \set{2,3}$, is a bounded,
polyhedral Lipschitz domain.~The~space~$\smash{W^{1,p}_0(\Omega)}$
is defined as the space of functions from $W^{1,p}(\Omega)$ whose trace vanishes on $\partial\Omega$.~We~equip $\smash{W^{1,p}_0(\Omega)}$ 
with the norm $\smash{\norm{\nabla\,\cdot\,}_p}$. 

We do not distinguish between spaces for scalar,
vector-~or~\mbox{tensor-valued}~functions. However, we always denote
vector-valued functions by boldface letters~and~tensor-valued
functions by capital boldface letters. The mean value of a locally
integrable function $f$ over a measurable set $M\subseteq \Omega$ is
denoted by
${\mean{f}_M\coloneqq \!\smash{\dashint_M f
    \,\textup{d}x}\coloneqq \!\smash{\frac 1 {|M|}\int_M f
    \,\textup{d}x}}$.~Moreover, we employ the notation
$\hskp{f}{g}\coloneqq \int_\Omega f g\,\textup{d}x$, whenever the
right-hand side is well-defined.

From the theory of Orlicz spaces (cf.~\cite{ren-rao}) and generalized
Orlicz spaces (cf.~\cite{HH19}) we use N-functions
$\psi \colon \setR^{\geq 0} \to \setR^{\geq 0}$ and generalized N-functions
$\psi \colon \Omega \times \setR^{\ge 0} \to \setR^{\ge 0}$, i.e.,
$\psi$ is a Carath\'eodory function such that $\psi(x,\cdot)$ is an
N-function for a.e.~$x \in \Omega$, respectively. The
modular~is~defined~via
$\rho_\psi(f)\coloneqq \rho_{\psi,\Omega}(f)\coloneqq \int_\Omega
\psi(\abs{f})\,\textup{d}x $ if $\psi$ is an N-function, and via
$\rho_\psi(f)\coloneqq \rho_{\psi,\Omega}(f)\coloneqq \int_\Omega
\psi(x,\abs{f(x)})\,\textup{d}x $, if $\psi$ is a generalized
N-function. An N-function
$\psi$~satisfies~the~\mbox{$\Delta_2$-condition} (in short,
$\psi \hspace*{-0.1em}\in\hspace*{-0.1em}\Delta_2$), if there exists
$K\hspace*{-0.1em}>\hspace*{-0.1em} 2$ such that for all
$t \hspace*{-0.1em}\ge\hspace*{-0.1em}
0$,~it~holds~${\psi(2\,t) \hspace*{-0.1em}\leq\hspace*{-0.1em} K\,
  \psi(t)}$. We denote the smallest such constant by
$\Delta_2(\psi)\hspace*{-0.1em}>\hspace*{-0.1em}0$. 
We need the following version of~the~\mbox{$\varepsilon$-Young} inequality: for every
${\varepsilon\!>\! 0}$,~there exits a constant $c_\epsilon\!>\!0 $, depending~only~on~$\Delta_2(\psi),\Delta_2( \psi ^*)\!<\!\infty$, such that~for~every~${s,t\!\geq\! 0}$, it holds
\begin{align}
  \label{ineq:young}
  \begin{split}
    t\,s&\leq \epsilon \, \psi(t)+ c_\epsilon \,\psi^*(s)\,.
  \end{split}
\end{align}

\subsection{Basic \hspace*{-0.1mm}properties \hspace*{-0.1mm}of \hspace*{-0.1mm}the \hspace*{-0.1mm}extra \hspace*{-0.1mm}stress \hspace*{-0.1mm}tensor}

\!\!Throughout~\hspace*{-0.1mm}the~\hspace*{-0.1mm}entire~\hspace*{-0.1mm}\mbox{paper}, we always assume that the extra stress tensor 
$\SSS$
has $(p,\delta)$-structure, which is defined here in a more stringent
way compared to Part I of the paper (cf.~\cite{kr-pnse-ldg-1}). A detailed
discussion and full proofs can be found, e.g., in
\cite{die-ett,dr-nafsa}. For a given tensor $\bfA\in \setR^{d\times d}$, we denote its symmetric part by
${\bfA^{\textup{sym}}\coloneqq \frac{1}{2}(\bfA+\bfA^\top)\in
  \setR^{d\times d}_{\textup{sym}}\coloneqq \{\bfA\in \setR^{d\times
    d}\mid \bfA=\bfA^\top\}}$.

For $p \in (1,\infty)$~and~$\delta\ge 0$, we define a special N-function
$\phi=\phi_{p,\delta}\colon\setR^{\ge 0}\to \setR^{\ge 0}$ by
\begin{align} 
  \label{eq:def_phi} 
  \varphi(t)\coloneqq  \int _0^t \varphi'(s)\, \mathrm ds,\quad\text{where}\quad
  \varphi'(t) \coloneqq  (\delta +t)^{p-2} t\,,\quad\textup{ for all }t\ge 0\,.
\end{align}
The properties of $\phi$ are discussed in detail in \cite{die-ett,dr-nafsa,kr-pnse-ldg-1}.
An important tool in our analysis play {\rm shifted N-functions}
$\{\psi_a\}_{\smash{a \ge 0}}$,~cf.~\cite{DK08,dr-nafsa}. For a given N-function $\psi\colon\mathbb{R}^{\ge 0}\to \mathbb{R}^{\ge
  0}$ we define the family  of shifted N-functions ${\psi_a\colon\mathbb{R}^{\ge
    0}\to \mathbb{R}^{\ge 0}}$,~${a \ge 0}$,~via
\begin{align}
  \label{eq:phi_shifted}
  \psi_a(t)\coloneqq  \int _0^t \psi_a'(s)\, \mathrm ds\,,\quad\text{where }\quad
  \psi'_a(t)\coloneqq \psi'(a+t)\frac {t}{a+t}\,,\quad\textup{ for all }t\ge 0\,.
\end{align}\newpage

\begin{assumption}[Extra stress tensor]\label{assum:extra_stress} We assume that the extra stress tensor $\SSS\colon\setR^{d\times d}\to \setR^{d\times d}_{\textup{sym}}$ belongs to $C^0(\setR^{d\times d}; \setR^{d\times d}_{\textup{sym}})\cap C^1(\setR^{d\times d}\setminus\{\mathbf{0}\}; \setR^{d\times d}_{\textup{sym}}) $ and satisfies $\SSS(\bfA)=\SSS(\bfA^{\textup{sym}})$ for all $\bfA\in \setR^{d\times d}$ and $\SSS(\mathbf{0})=\mathbf{0}$. Moreover, we assume~that~the~tensor $\SSS=(S_{ij})_{i,j=1,\dots,d}$ has $(p,\delta)$-structure, i.e.,
  for some $p \in (1, \infty)$, $ \delta\in [0,\infty)$, and the
  N-function $\varphi=\varphi_{p,\delta}$ (cf.~\eqref{eq:def_phi}), there
  exist constants $C_0, C_1 >0$ such that
   \begin{align}
       \sum\limits_{i,j,k,l=1}^d \partial_{kl} S_{ij} (\bfA)
       B_{ij}B_{kl} &\ge C_0 \, \frac{\phi'(|\bfA^{\textup{sym}}|)}{|\bfA^{\textup{sym}}|}\,|\bfB^{\textup{sym}}|^2\,,\label{assum:extra_stress.1}
       \\
       \big |\partial_{kl} S_{ij}({\bfA})\big | &\le C_1 \, \frac{\phi'(|\bfA^{\textup{sym}}|)}{|\bfA^{\textup{sym}}|}\label{assum:extra_stress.2}
   \end{align}
are satisfied for all $\bfA,\bfB \in \setR^{d\times d}$ with $\bfA^{\textup{sym}}\neq \mathbf{0}$ and all $i,j,k,l=1,\dots,d$.~The~constants $C_0,C_1>0$ and $p\in (1,\infty)$ are called the {characteristics} of $\SSS$.
\end{assumption}

\begin{remark}
  {\rm (i) It is well-known (cf.~\cite{dr-nafsa}) that the conditions
    \eqref{assum:extra_stress.1}, \eqref{assum:extra_stress.2} imply~the conditions in the definition of the $(p,\delta)$-structure in
    Part~I~of~the paper (cf.~\cite{kr-pnse-ldg-1}).

    (ii)     Assume that $\SSS$ satisfies Assumption \ref{assum:extra_stress} for some
    $\delta \in [0,\delta_0]$.~Then,~if~not otherwise stated, the
    constants in the estimates depend only on the characteristics~of~$\SSS$~and on $\delta_0\ge 0$, but are independent of $\delta\ge 0$.

    (iii)     Let $\phi$ be defined in \eqref{eq:def_phi} and 
    $\{\phi_a\}_{a\ge 0}$ be the corresponding family of the shifted \mbox{N-functions}. Then, the operators 
    $\SSS_a\colon\mathbb{R}^{d\times d}\to \smash{\mathbb{R}_{\textup{sym}}^{d\times
      d}}$, $a \ge 0$, defined, for every $a \ge 0$
    and~$\bfA \in \mathbb{R}^{d\times d}$, via 
\begin{align}
  \label{eq:flux}
  \SSS_a(\bfA) \coloneqq 
  \frac{\phi_a'(\abs{\bfA^{\textup{sym}}})}{\abs{\bfA^{\textup{sym}}}}\,
  \bfA^{\textup{sym}}\,, 
\end{align}
have $(p, \delta +a)$-structure.  In this case, the characteristics of
$\SSS_a$ depend~only~on~${p\in (1,\infty)}$ and are independent of
$\delta \geq 0$ and $a\ge 0$.
  }
\end{remark}

Closely related to the extra stress tensor $\SSS$ with
$(p,\delta)$-structure are the functions
$\bfF,\bfF^*\colon\setR^{d\times d}\to \setR^{d\times d}_{\textup{sym}}$, 
defined, for every $\bfA\in \mathbb{R}^{d\times d}$, via
\begin{align}
\begin{aligned}
    \bfF(\bfA)&\coloneqq (\delta+\vert \bfA^{\textup{sym}}\vert)^{\smash{\frac{p-2}{2}}}\bfA^{\textup{sym}}\,,\\
    \bfF^*(\bfA)&\coloneqq (\delta^{p-1}+\vert \bfA^{\textup{sym}}\vert)^{\smash{\frac{p'-2}{2}}}\bfA^{\textup{sym}}\,.
    \end{aligned}\label{eq:def_F}
\end{align}
The connection between
$\SSS,\bfF,\bfF^* \hspace{-0.05em}\colon\hspace{-0.05em}\setR^{d \times d}
\hspace{-0.05em}\to\hspace{-0.05em} \setR^{d\times d}_{\textup{sym}}$ and
$\phi_a,(\phi_a)^*\hspace{-0.05em}\colon\hspace{-0.05em}\setR^{\ge
  0}\hspace{-0.05em}\to\hspace{-0.05em} \setR^{\ge
  0}$,~${a\hspace{-0.05em}\ge\hspace{-0.05em} 0}$, is best explained
by the following result (cf.~\cite{die-ett,dr-nafsa,dkrt-ldg}).

\begin{proposition}
  \label{lem:hammer}
  Let $\SSS$ satisfy Assumption~\ref{assum:extra_stress}, let $\varphi$ be defined in \eqref{eq:def_phi}, and let $\bfF,\bfF^*$ be defined in \eqref{eq:def_F}. Then, uniformly with respect to 
  $\bfA, \bfB \in \setR^{d \times d}$, we have that\vspace{-1mm}
    \begin{align}\label{eq:hammera}
        \begin{aligned}
        \big(\SSS(\bfA) - \SSS(\bfB)\big)
      :(\bfA-\bfB ) &\sim  \abs{ \bfF(\bfA) - \bfF(\bfB)}^2
      \\
      &\sim \phi_{\abs{\bfA^{\textup{sym}}}}(\abs{\bfA^{\textup{sym}}
        - \bfB^{\textup{sym}}})
      \\
      &\sim(\varphi_{\abs{\bfA^{\textup{sym}} }})^*(\abs{\SSS(\bfA ) - \SSS(\bfB )})
      \\&\sim (\phi^*) _{|\SSS ( \bfA^{\textup{sym}})|} (|\SSS (\bfA) - \SSS ( \bfB) |)
      \,,
      \end{aligned}
    \end{align}
    \vspace{-6.5mm}
    
    \begin{align}
         \mspace{-13mu}\smash{\abs{ \bfF^*(\bfA) - \bfF^*(\bfB)}^2}
      \label{eq:hammerf}
      &\sim
        \smash{\smash{\big(\phi^*\big)}_{\smash{\abs{\bfA^{\textup{sym}}}}}(\abs{\bfA^{\textup{sym}}
        - \bfB^{\textup{sym}}})} \,,
      \\[2mm]
       \label{eq:F-F*3}
     \mspace{-13mu}\smash{\abs{\bfF^*(\SSS(\bfB))-\bfF^*(\SSS(\bfA))}^2}
    &\sim  \smash{\abs{\bfF(\bfB)-\bfF(\bfA)}^2}\,.
    \end{align}
  The constants in \eqref{eq:hammera} 
  depend only on the characteristics of ${\SSS}$.
\end{proposition} 
\begin{remark}\label{rem:sa}
  {\rm
For the operators $\SSS_a\hspace{-0.1em}\colon\hspace{-0.1em}\mathbb{R}^{d\times d}\hspace{-0.1em}\to\hspace{-0.1em}\smash{\mathbb{R}_{\textup{sym}}^{d\times
      d}}$, $a \ge  0$, defined~in~\eqref{eq:flux},~the~asser-tions of Proposition \ref{lem:hammer} hold with $\phi\colon\mathbb{R}^{\ge 0}\to \mathbb{R}^{\ge 0}$ replaced
by $\phi_a\colon\mathbb{R}^{\ge 0}\to \mathbb{R}^{\ge 0}$, $a\ge 0$.}
\end{remark}

The following results can be found in~\cite{DK08,dr-nafsa}.\newpage

\begin{lemma}[Change of Shift]\label{lem:shift-change}
    Let $\varphi$ be defined in \eqref{eq:def_phi} and let $\bfF$ be defined in \eqref{eq:def_F}. Then,
  for each $\varepsilon>0$, there exists $c_\varepsilon\geq 1$ (depending only
  on~$\varepsilon>0$ and the characteristics of $\phi$) such that for every $\bfA,\bfB\in\smash{\setR^{d \times d}_{\textup{sym}}}$ and $t\geq 0$, it holds
  \begin{align*}
    \smash{\phi_{\abs{\bfB}}(t)}&\leq \smash{c_\varepsilon\, \phi_{\abs{\bfA}}(t)
    +\varepsilon\, \abs{\bfF(\bfB) - \bfF(\bfA)}^2\,,}
    \\
        \smash{\phi_{\abs{\bfB}}(t)}&\leq \smash{c_\varepsilon\, \phi_{\abs{\bfA}} (t)
    +\varepsilon\, \phi_{\abs{\bfA}}\big(\bigabs{\abs{\bfB} - \abs{\bfA}}\big)\,,}
    \\
    \smash{(\phi_{\abs{\bfB}})^*(t)}&\leq \smash{c_\varepsilon\, (\phi_{\abs{\bfA}})^*(t)
                                      +\varepsilon\, \abs{\bfF(\bfB) - \bfF(\bfA)}^2}\,,
   \\
    \smash{(\phi_{\abs{\bfB}})^*(t)}&\leq \smash{c_\varepsilon\, (\phi_{\abs{\bfA}})^*(t)
    +\varepsilon\, \phi_{\abs{\bfA}}\big(\bigabs{\abs{\bfB} - \abs{\bfA}}\big)}\,.
  \end{align*}
\end{lemma}

\subsection{The $p$-Navier--Stokes system} 
Let us briefly recall some well-known facts about the $p$-Navier--Stokes system
 \eqref{eq:p-navier-stokes}. For $p\in (1,\infty)$, we define the function spaces\\[-4.5mm]
\begin{align*}
     \Vo\coloneqq (W^{1,p}_0(\Omega))^d\,,\qquad
  \Qo\coloneqq L_0^{p'}(\Omega)\coloneqq \big\{f\in
  L^{p'}(\Omega)\;|\;\mean{f}_{\Omega}=0\big\}\,.\\[-6mm] 
\end{align*}
    With this particular notation, the weak formulation of problem \eqref{eq:p-navier-stokes} is the following:
    
\textbf{Problem (Q).} For given $\bfg\in L^{p'}(\Omega)$, find $(\bfv,q)\in \Vo \times \Qo$ such that  for all $(\bfz,z)^\top\in \Vo\times Q $, it holds
\begin{align}
    (\SSS(\bfD\bfv),\bfD\bfz)+([\nabla\bfv]\bfv,\bfz)-(q,\divo\bfz)&=(\bfg,\bfz)\label{eq:q1}\,,\\
    (\divo\bfv,z)&=0\label{eq:q2}\,.
\end{align}

Alternatively, we can reformulate Problem (Q) ``hidding'' the pressure.

\textbf{Problem (P).} \hspace{-0.1mm}For \hspace{-0.1mm}given \hspace{-0.1mm}$\bfg\!\in \!L^{p'}(\Omega)$, \hspace{-0.1mm}find \hspace{-0.1mm}$\bfv\!\in\! \Vo(0)$ 
\hspace{-0.1mm}such \hspace{-0.1mm}that \hspace{-0.1mm}for \hspace{-0.1mm}all \hspace{-0.1mm}${\bfz\!\in\! \Vo(0)}$,~\hspace{-0.1mm}it~\hspace{-0.1mm}holds
\begin{align}
    (\SSS(\bfD\bfv),\bfD\bfz)+([\nabla\bfv]\bfv,\bfz)&=(\bfg,\bfz)\,,\label{eq:p}
\end{align}
where $\Vo(0)\coloneqq \{\bfz\in \Vo\mid \divo \bfz=0\}$.

The theory of pseudo-monotone operators yields the existence of a weak
solution of Problem (P) for $p>\frac{3d}{d+2}$ (cf.~\cite{lions-quel}). DeRham's lemma, the solvability of the divergence
equation, and the negative norm theorem,~then, ensure the solvability
of Problem (Q). {There holds the following regularity property of the
pressure if the velocity satisfies a natural regularity assumption.
\begin{lemma}\label{lem:pres}
  Let $\SSS$ satisfy Assumption~\ref{assum:extra_stress} with
  $p\hspace{-0.1em}\in \hspace{-0.1em}(2,\infty)$ and $\delta \hspace{-0.1em}> \hspace{-0.1em}0$, let~${\bfg\hspace{-0.1em}\in \hspace{-0.1em}
  L^{p'}(\Omega)}$, and let $(\bfv,q)^\top \in \Vo(0)\times \Qo$
  be a weak solution of Problem (Q) with $\bfF(\bfD \bfv) \in
  W^{1,2}(\Omega)$. Then, it holds $\nabla q \!\in \! L^{p'}(\Omega)$.
  If, in addition, $\bfg\!\in \! L^{2}(\Omega)$, then~${(\delta
    +\abs{\bfD\bfv})^{2-p}\abs{\nabla q}^2 \!\in \! L^1(\Omega)}$.
\end{lemma}
\begin{proof}
  Using an algebraic identity connecting $\nabla^2 \bfv$ and $\nabla
  \bfD \bfv$ (cf.~\cite[Lemma~6.3]{mnr3}), $p>2$, and
  \cite[Lemma~3.8]{bdr-7-5} we obtain
  \begin{align}     \label{eq:embed1}
    \begin{aligned}
      \abs{\nabla^2\bfv}^2\le c\, \abs{\nabla\bfD\bfv}^2&\le
      c\,\delta^{2-p}\,\abs{\nabla\bfF(\bfD\bfv)}^2 \,,
      \\
      \abs{\nabla\bfv}^2&\le \delta^{2-p}\,\abs{\bfF(\bfD\bfv)}^2 \,.
    \end{aligned}
  \end{align}
  Thus, $\bfF (\bfD\bfv) \in W^{1,2}(\Omega)$ implies
  $\bfv \in W^{2,2}(\Omega) \cap W^{1,2}_0(\Omega)\hookrightarrow
  L^\infty (\Omega)$. In consequence, the convective term belongs to
  $W^{1,2}(\Omega)$. Using
  $\abs{\nabla \bfS(\bfD\bfv)}^2 \sim \abs{\nabla \bfF(\bfD\bfv)}^2\,
  (\delta+ \abs{\bfD\bfv})^{p-2}$
  (cf.~\cite[Proposition~2.4]{br-plasticity}) and H\"older's
  inequality we obtain that $\divo \bfS(\bfD\bfv) \in
  L^{p'}(\Omega)$. Thus, the weak formulation \eqref{eq:q1} and
  $\bfg \in L^{p'}(\Omega)$ yield that $\nabla q \in
  L^{p'}(\Omega)$. This proves the first assertion. If, in addition
  $\bfg \in L^2(\Omega)$, the above also implies that
  \eqref{eq:p-navier-stokes} holds a.e.~in $\Omega$, which implies
  a.e.~in $\Omega$
  \begin{align*}%    \label{eq:est}
    \begin{aligned}
      \abs{\nabla q}^2 &\le
      c\,\abs{\nabla \bfS(\bfD\bfv)}^2 + c\,\norm{\bfv}_\infty ^2 \,
      \abs{\nabla \bfv}^2+c\,\abs{\bfg}^2
      \\
      &\le c\,\abs{\nabla \bfF(\bfD\bfv)}^2 \, (\delta+
      \abs{\bfD\bfv})^{p-2}+ c\, \delta^{2-p}\, 
      \abs{\bfF(\bfD\bfv)}^2 +c\,\abs{\bfg}^2\,.
    \end{aligned}
  \end{align*}
  Dividing by $(\delta+ \abs{\bfD\bfv})^{p-2}$ and using $p>2$ yields
  the second assertion. 
\end{proof}}
\subsection{DG spaces, jumps and averages}\label{sec:dg-space}

\subsubsection{\!Triangulations}

\!\!We \hspace{-0.1mm}always \hspace{-0.1mm}denote \hspace{-0.1mm}by \hspace{-0.1mm}$\mathcal{T}_h$, \hspace{-0.1mm}$h\!>\!0$,\hspace{-0.1mm} a \hspace{-0.1mm}family \hspace{-0.1mm}of \hspace{-0.1mm}uniformly~\hspace{-0.1mm}shape
regular and conforming triangulations
of~${\Omega\hspace{-0.1em}\subseteq \hspace{-0.1em}\setR^d}$,~${d\hspace{-0.1em}\in\hspace{-0.1em} \set{2,3}}$,~cf.~\cite{BS08},~each
consisting of \mbox{$d$-dimensional} simplices $K$.  The parameter
$h\hspace{-0.1em}>\hspace{-0.1em}0$, refers to the~maximal~\mbox{mesh-size}~of~$\mathcal{T}_h$, for \hspace{-0.1mm}which
\hspace{-0.1mm}we \hspace{-0.1mm}assume \hspace{-0.1mm}for \hspace{-0.1mm}simplicity \hspace{-0.1mm}that \hspace{-0.1mm}$h \!\le\! 1$. \!Moreover, \hspace{-0.1mm}we \hspace{-0.1mm}assume~\hspace{-0.1mm}that~\hspace{-0.1mm}the~\hspace{-0.1mm}\mbox{chunkiness}~\hspace{-0.1mm}is bounded by some constant $\omega_0\!>\!0$, independent on
$h$. \!By $\Gamma_h^{i}$,~we~denote~the~interior~faces, and put
$\Gamma_h\coloneqq  \Gamma_h^{i}\cup \partial\Omega$.  We
assume~that~each~simplex~${K \in \mathcal{T}_h}$ has at most~one~face from $\partial\Omega$.  We introduce the following scalar
products~on~$\Gamma_h$
\begin{align*}
  \skp{f}{g}_{\Gamma_h} \coloneqq  \smash{\sum_{\gamma \in \Gamma_h} {\langle f, g\rangle_\gamma}}\,,\quad\text{ where }\quad\langle f, g\rangle_\gamma\coloneqq \int_\gamma f g \,\textup{d}s\quad\text{ for all }\gamma\in \Gamma_h\,,
\end{align*}
if all the integrals are well-defined. Similarly, we define the products 
$\smash{\skp{\cdot}{\cdot}_{\partial\Omega}}$ and~$\smash{\skp{\cdot}{\cdot}_{\Gamma_h^{i}}}$. We extend the notation of
modulars to the sets $\smash{\Gamma_h^{i}}$, 
$\partial \Omega$, and $\smash{\Gamma_h}$, i.e., we
define the modulars ${\rho_{\psi,B}(f)\coloneqq  \smash{\int _B
\psi(\abs{f})\,\textup{d}s}}$ for every $f \in \smash{L^\psi(B)}$, where $B= \smash{\Gamma_h^{i}}$~or~${B=\partial \Omega}$~or~${B=\smash{\Gamma_h}}$.

\subsubsection{Broken \hspace*{-0.1mm}function \hspace*{-0.1mm}spaces and \hspace*{-0.1mm}projectors}
\hspace*{-0.1mm}For \hspace*{-0.1mm}every
\hspace*{-0.1mm}$m \hspace{-0.12em}\in\hspace{-0.12em}
\setN_0$~\hspace*{-0.1mm}and~\hspace*{-0.1mm}${K\hspace{-0.12em}\in\hspace{-0.12em} \mathcal{T}_h}$, we
denote by ${\mathcal P}_m(K)$, the space of polynomials of
degree at most $m$ on $K$. Then, for given $k \in \setN_0$ and $p\in
(1,\infty)$,  
we define the spaces
\begin{align}
  \begin{split}
    Q_h^k&\coloneqq \big\{ q_h\in L^1(\Omega)\,\mid q_h|_K\in \mathcal{P}_k(K)\text{ for all }K\in \mathcal{T}_h\big\}\,,\\
    V_h^k&\coloneqq \big\{\bfv_h\in L^1(\Omega)^d\,\mid \bfv_h|_K\in \mathcal{P}_k(K)^d\text{ for all }K\in \mathcal{T}_h\big\}\,,\\
    X_h^k&\coloneqq \big\{\bfX_h\in L^1(\Omega)^{d\times d}\,\mid \bfX_h|_K\in \mathcal{P}_k(K)^{d\times d}\text{ for all }K\in \mathcal{T}_h\big\}\,,\\
        W^{1,p}(\mathcal T_h)&\coloneqq \big\{\bfw_h\in L^1(\Omega)^d\mid \bfw_h|_K\in W^{1,p}(K)^d\text{ for all }K\in \mathcal{T}_h\big\}\,.
  \end{split}\label{eq:2.19}
\end{align}
In addition, for given $k \hspace{-0.1em}\in\hspace{-0.1em} \setN_0$, we set $\Qhkc\hspace{-0.1em}\coloneqq \hspace{-0.1em} Q_h^k\cap C^0(\overline{\Omega})$.
Note~that~${W^{1,p}(\Omega)\hspace{-0.1em}\subseteq\hspace{-0.1em} \WDG}$ and
$V_h^k\hspace{-0.1em}\subseteq \hspace{-0.1em}\WDG$. We denote by ${\PiDG \hspace{-0.1em}\colon\hspace{-0.1em} L^1(\Omega)\hspace{-0.1em}\to\hspace{-0.1em} V_h^k}$, the (local)
$L^2$--projection~into~$V_h^k$, which for every $\bfv \in
L^1(\Omega)$ and $\bfz_h
\in V_h^k$ is~defined~via  $\bighskp{\PiDG \bfv}{\bfz_h}=\hskp{\bfv}{\bfz_h}$. 
Analogously, we define the (local)
$L^2$--projection into $X_h^k$, i.e., ${\PiDG\colon L^1(\Omega) \to \Xhk}$.

For every  $\bfw_h\in \WDG$, we denote by $\nabla_h \bfw_h\in L^p(\Omega)$,
the local gradient,~defined via
$(\nabla_h \bfw_h)|_K\coloneqq \nabla(\bfw_h|_K)$
for~all~${K\in\mathcal{T}_h}$.  For every $\bfw_h\in \WDG$ and
interior faces $\gamma\in \Gamma_h^{i}$ shared by adjacent elements
$K^-_\gamma, K^+_\gamma\in \mathcal{T}_h$, we~denote~by
\begin{align}
  \{\bfw_h\}_\gamma&\coloneqq \smash{\frac{1}{2}}\big(\textup{tr}_\gamma^{K^+}(\bfw_h)+
  \textup{tr}_\gamma^{K^-}(\bfw_h)\big)\in
  L^p(\gamma)\,, \label{2.20}\\
  \llbracket\bfw_h\otimes\bfn\rrbracket_\gamma
  &\coloneqq \textup{tr}_\gamma^{K^+}(\bfw_h)\otimes\bfn^+_\gamma+
    \textup{tr}_\gamma^{K^-}(\bfw_h)\otimes\bfn_\gamma^- 
    \in L^p(\gamma)\,,\label{eq:2.21}
\end{align}
the {average} and {normal jump}, resp., of $\bfw_h$ on $\gamma$.
Moreover,  for boundary faces $\gamma\in \partial\Omega$, we define boundary averages and 
boundary~jumps,~resp.,~via
\begin{align}
  \{\bfw_h\}_\gamma&\coloneqq \textup{tr}^\Omega_\gamma(\bfw_h) \in L^p(\gamma)\,,\label{eq:2.23a} \\
  \llbracket \bfw_h\otimes\bfn\rrbracket_\gamma&\coloneqq 
  \textup{tr}^\Omega_\gamma(\bfw_h)\otimes\bfn \in L^p(\gamma)\,,\label{eq:2.23} 
\end{align}
where $\bfn\colon\partial\Omega\to \mathbb{S}^{d-1}$ denotes the unit normal vector field to $\Omega$ pointing outward. 
Analogously, we
define $\{\bfX_h\}_\gamma$ and $ \llbracket\bfX_h\bfn\rrbracket_\gamma
$~for all $\bfX_h \in \Xhk$ and $\gamma\in \Gamma_h$. Furthermore, if there is no
danger~of~confusion, then we will omit the index $\gamma\in \Gamma_h$,~in~particular,  if we interpret jumps and averages as global functions defined on whole $\Gamma_h$.

\subsubsection{DG gradient and jump operators}

For every $k\in \mathbb{N}_0$ and  face $\gamma\in \Gamma_h$, we define the
\textbf{(local)~jump~\mbox{operator}}
$\smash{\boldsymbol{\mathcal{R}}_{h,\gamma}^k \colon\WDG \to X_h^k}$~for~every~${\bfw_h\in \smash{\WDG}}$ (using
Riesz representation)~via $(\boldsymbol{\mathcal{R}}_{h,\gamma}^k\bfw_h,\bfX_h)\coloneqq \langle \llbracket\bfw_h\otimes\bfn\rrbracket_\gamma,\{\bfX_h\}_\gamma\rangle_\gamma$ for all $\bfX_h\in X_h^k$. 
For every $k\in \mathbb{N}_0$, the \textbf{(global) jump operator} $\smash{\Rhk\coloneqq \sum_{\gamma\in \Gamma_h}{\boldsymbol{\mathcal{R}}_{\gamma,h}^k}\colon\WDG \to X_h^k}$, 
by definition, for every $\bfw_h\in \smash{\WDG}$ and  $\bfX_h\in X_h^k$ satisfies
\begin{align}
  \smash{\big(\Rhk\bfw_h,\bfX_h\big)=\big\langle
  \llbracket\bfw_h\otimes\bfn\rrbracket,\{\bfX_h\}\big\rangle_{\Gamma_h}\,.}\label{eq:2.25.1}
\end{align}
Then, \hspace*{-0.1mm}for \hspace*{-0.1mm}every \hspace*{-0.1mm}$k\!\in\! \mathbb{N}_0$, \hspace*{-0.1mm}the \hspace*{-0.1mm}\textbf{DG \hspace*{-0.1mm}gradient \hspace*{-0.1mm}operator} 
\hspace*{-0.1mm}$  {\Ghk\coloneqq \!\nabla_h\!-\!\Rhk\colon \!\WDG\!\to\! L^p(\Omega)}$,  
for every $\bfw_h\in \smash{\WDG}$ and $\bfX_h\in X_h^k$ satisfies
\begin{align}
\smash{\big(\Ghk\bfw_h,\bfX_h\big)=(\nabla_h\bfw_h,\bfX_h)
  -\big\langle \llbracket
  \bfw_h\otimes\bfn\rrbracket,\{\bfX_h\}\big\rangle_{\Gamma_h}
\,. } \label{eq:DGnablaR1}
\end{align} 
Apart from that, for every $\bfw_h\in \smash{\WDG}$, we introduce the {DG norm} as
\begin{align}
    \smash{\|\bfw_h\|_{\nabla,p,h}\coloneqq \|\nabla_h\bfw_h\|_p+h^{\frac{1}{p}}\big\|h^{-1}\jump{\bfw_h\otimes \bfn}\big\|_{p,\Gamma_h}\,.}
\end{align}
Owing to \cite[(A.26)--(A.28)]{dkrt-ldg}, there exists a constant $c>0$ such that~for every $\bfw_h\in \smash{\WDG}$, it holds
\begin{align}\label{eq:eqiv0}
    \smash{c^{-1}\,\|\bfw_h\|_{\nabla,p,h}\leq \big\|\Ghnk\bfw_h\big\|_p+h^{\frac{1}{p}}\big\|h^{-1}\jump{\bfw_h\otimes \bfn}\big\|_{p,\Gamma_h}\leq c\,\|\bfw_h\|_{\nabla,p,h}\,.}
\end{align}
For an  N-function $\psi$, we define the pseudo-modular\footnote{{The
  definition of an pseudo-modular can be found in \cite{Mu}.} We
  extend the notion of DG Sobolev spaces to  DG Sobolev-Orlicz
  spaces $W^{1,\psi}(\mathcal T_h)\coloneqq \big\{\bfw_h\!\in \!L^1(\Omega)\mid \bfw_h|_K\!\in\! W^{1,\psi}(K)\text{ for all }K\in \mathcal{T}_h\big\}$.} $\smash{m_{\psi,h}}\colon W^{1,\psi}(\mathcal T_h)\to \mathbb{R}^{\ge 0}$ for every $\bfw_h\in W^{1,\psi}(\mathcal T_h)$ via
  \begin{align} \label{def:mh}
    m_{\psi,h}(\bfw_h)&\coloneqq  h\,\rho_{\psi,\smash{\Gamma_h}}\big(h^{-1}\jump{\bfw_h\otimes \bfn}\big)\,.
  \end{align}
For $\psi = \phi_{p,0}$, we have that
$m_{\psi,h}(\bfw_h)=h\,\|h^{-1}\jump{\bfw_h\otimes
  \bfn}\|_{p,\Gamma_h}^p $~for~all~${\bfw_h\in W^{1,\psi}(\mathcal T_h)}$.
  
\subsubsection{Symmetric \hspace*{-0.1mm}DG \hspace*{-0.1mm}gradient \hspace*{-0.1mm}and \hspace*{-0.1mm}symmetric \hspace*{-0.1mm}jump \hspace*{-0.1mm}operators}

\!For~\hspace*{-0.1mm}\mbox{every} ${\bfw_h\in \WDG}$, we denote by $\bfD_h\bfw_h\coloneqq \![\nabla_h\bfw_h]^{\textup{sym}}\!\in\!
L^p(\Omega;\mathbb{R}^{d\times d}_{\textup{sym}})$ the~local~\mbox{symmetric}  gradient.  In addition, for every
$k\in \setN_0$ and
$\smash{X_h^{\smash{k,\textup{sym}}}\coloneqq X_h^k\cap
  L^p(\Omega;\mathbb{R}^{d\times d}_{\textup{sym}})}$,~we~define~the
{symmetric DG gradient
  operator}~$ \smash{\Dhk\colon \!\WDG\!\to\!
  L^p(\Omega;\mathbb{R}^{d\times
    d}_{\textup{sym}})}$,\hspace{.1em}for\hspace{.1em}every\hspace{.1em}${\bfw_h\!\in \!\WDG}$, via
$\smash{\Dhk\bfw_h\coloneqq [\Ghk\bfw_h]^{\textup{sym}}\!\in\!
  L^p(\Omega;\mathbb{R}^{d\times d}_{\textup{sym}})}$, i.e., for
every  
$\bfX_h\in \smash{X_h^{k,\textup{sym}}}$, we have
that
	\begin{align}
	\smash{	\big(\Dhk\bfw_h,\bfX_h\big)
		=(\bfD_h\bfw_h,\bfX_h)
		-\big\langle \llbracket \bfw_h\otimes\bfn\rrbracket,\{\bfX_h\}\big\rangle_{\Gamma_h}\,.}\label{eq:2.24}
	\end{align}
	Apart from that, for every $\bfw_h\in \WDG$, we introduce the {symmetric DG norm} as
	\begin{align}
	\smash{	\|\bfw_h\|_{\bfD,p,h}\coloneqq \|\bfD_h\bfw_h\|_p
		+\smash{h^{\frac{1}{p}}\big\|  h^{-1} \llbracket\bfw_h\otimes\bfn\rrbracket\big\|_{p,\Gamma_h}}\,.}\label{eq:2.29}
	\end{align}
	
	A discrete Korn inequality on $V_h^k$ plays an important role
        in the numerical analysis of the
        $p$-Navier--Stokes~system~\eqref{eq:p-navier-stokes}.
	 \begin{proposition}[Discrete Korn inequality]\label{korn}
	    For every $p\in (1,\infty)$ and $k\in \setN$, there
		exists a constant~${c_{\mathbf{Korn}}>0}$ such that 
		for every $\bfv_h\in V_h^k$, it holds
		\begin{align}
		\smash{\|\bfv_h\|_{\nabla,p,h}\leq c_{\mathbf{Korn}}\,\|\bfv_h\|_{\bfD,p,h}}\label{eq:equi1}\,.
		\end{align}
	\end{proposition}

	\begin{proof}
          See \cite[Proposition 2.4]{kr-pnse-ldg-1}.
        \end{proof}
        
        We also need a version of Korn's inequality
        involving a function from $\WDG$.

    \begin{proposition}[Korn type inequality]\label{prop:kornii}
      For every $p\in (1,\infty)$ and $k\in \setN$, there exists a
      constant~${c}>0$ such that for every
      $\bfv_h\in V_h^k$ and every  $\bfw_h\in W^{2,p}(\mathcal{T}_h)$, it holds
    \begin{align*}
        \smash{\|\bfw_h-\bfv_h\|_{\nabla,p,h}^p\leq c\,\|\bfw_h-\bfv_h\|_{\bfD,p,h}^p+c\,h^p\,\|\nabla_h^2\bfw_h\|_p^p\,.}
    \end{align*}
    \end{proposition}
    \begin{proof} 
    Using that
      $\bfw_h-\PiDG\bfw_h=\bfw_h-\PiDG\bfw_h-\PiDG(\bfw_h-\PiDG\bfw_h)$,
      the~trace~inequality \cite[(A.21)]{kr-phi-ldg} and the
      approximation property of $\PiDG$ (cf.~\cite[(A.14)]{kr-phi-ldg}) (each for $\psi=\varphi_{p,0}$), we obtain
       \begin{align}
        \label{cor:e5.1}
        \smash{\big\|\bfw_h-\PiDG\bfw_h\big\|_{\nabla,p,h}^p\leq c\, h^p\|\nabla_h^2\bfw_h\|_p^p\,.}
      \end{align}
      Using the discrete Korn inequality \eqref{eq:equi1} in conjunction with  \eqref{cor:e5.1}, we find that
        \begin{align*}
          \|\bfw_h-\bfv_h\|_{\nabla,p,h}^p
          &\leq c\,\big\|\PiDG\bfw_h-\bfv_h\big\|_{\nabla,p,h}^p+c\,\big\|\bfw_h-\PiDG\bfw_h\big\|_{\nabla,p,h}^p
          \\[-0.5mm]
          &\leq c\,\big\|\PiDG\bfw_h-\bfv_h\big\|_{\bfD,p,h}^p+c\, h^p\|\nabla_h^2\bfw_h\|_p^p
          \\[-0.5mm]&\leq c\,\|\bfw_h-\bfv_h\|_{\bfD,p,h}^p+c\,\big\|\bfw_h-\PiDG\bfw_h\big\|_{\nabla,p,h}^p+c\, h^p\|\nabla_h^2\bfw_h\|_p^p 
          \\[-0.5mm]&\leq c\,\|\bfw_h-\bfv_h\|_{\bfD,p,h}^p+c\, h^p\|\nabla_h^2\bfw_h\|_p^p \,.
        \end{align*}
    \end{proof}
	
	For the symmetric DG norm, there holds a similar relation like \eqref{eq:eqiv0}.
	
	\begin{proposition}\label{prop:equivalences}
          For every $p\in (1,\infty)$
          and $k\in \setN$, there exists
          a constant~${c>0}$
          such that for every $\bfw_h\in \WDG$, it holds
		\begin{align}
		\begin{aligned}
		\smash{c^{-1}\,\|\bfw_h\|_{\bfD,p,h}
		\leq \big\|\Dhk\bfw_h\big\|_p
		+h^{\frac{1}{p}}\big\|  h^{-1}\llbracket\bfw_h\otimes\bfn\rrbracket\big\|_{p,\Gamma_h}
		\leq c\,\|\bfw_h\|_{\bfD,p,h}\,.}
		\end{aligned}\label{eq:equi2}
		\end{align}
	\end{proposition}
	\begin{proof}
		See \cite[Proposition 2.5]{kr-pnse-ldg-1}.
	\end{proof}
	
\subsubsection{\!\!\!DG \hspace*{-0.15mm}divergence \hspace*{-0.15mm}operator}
	\!\!For \hspace{-0.15mm}every \hspace{-0.15mm}$\bfw_h\!\in\!\WDG$, \hspace{-0.1mm}we \hspace{-0.15mm}denote~\hspace{-0.15mm}by~\hspace{-0.15mm}$\textup{div}_h\bfw_h$ $\coloneqq  \text{tr}(\nabla_h \bfw_h)\in L^p(\Omega)$, the {local divergence}. 
	In addition, for every $k\in \setN_0$,~the~{DG~divergence~operator} 
	$\Divhk\colon\WDG\to L^p(\Omega)$, for every $\bfw_h\in \WDG$, is~defined~via $\smash{\Divhk\bfw_h\coloneqq \text{tr}(\Ghk\bfw_h)
		=\text{tr}(\Dhk\bfw_h)\in Q_h^k}$,~i.e., 
	for every %$\bfw_h\in \WDG$~and~
	$z_h  \in \smash{Q_h^k}$, we~have~that
	\begin{align*}
	    \smash{\big(\Divhk\bfw_h,z_h\big)%&
	    =(\textup{div}_h\bfw_h,z_h)
		-\langle \llbracket \bfw_h\cdot\bfn\rrbracket,\{z_h\}\rangle_{\Gamma_h}\,.}%\\
	%	&=-(\bfw_h,\nabla_h z_h)
	%	+\big\langle \{\bfw_h\cdot\bfn\}, \llbracket z_h\rrbracket\big\rangle_{\Gamma_h^{i}}\,,
	\end{align*}
    %i.e., $(\Divhk \bfw_h,z_h)=-( \bfw_h,\nabla z_h)$ if $z_h  \in \Qhkc$. 
    Therefore,  for every $\bfv\in W^{1,p}_0(\Omega)$~and~$z_h\in \Qhkc$, we have that
\begin{align}
  \label{eq:div-dg}
  \begin{aligned}
    \smash{\big(\Divhk \PiDG \bfv,z_h\big)=-( \bfv,\nabla z_h)=(\divo\bfv, z_h)\,.}
\end{aligned}
\end{align}

\section{Fluxes and LDG formulations}\label{sec:ldg}
	
To obtain the LDG formulation of~\eqref{eq:p-navier-stokes}~for
$k \in \setN$, we proceed as in Part I \cite[Sec. 3]{kr-pnse-ldg-1} to get
the discrete counterpart~of~\mbox{Problem}~(Q). 
   Recall that, {restricting ourselves to the case that $q_h\in \Qhkco\coloneqq\Qhko\cap C^0(\overline{\Omega})$},~the~numerical fluxes are, for every stabilization
   parameter $\alpha>0$, defined~via
   \begin{align}
   	\label{def:flux-v1}
   	\flux{\bfv}_{h,\sigma}(\bfv_h) &\coloneqq  
   	\begin{cases}
   		\avg{\bfv_h} &\text{on $\Gamma_h^{i}$}
   		\\
   		\bfv^* &\text{on $\partial\Omega$}
   	\end{cases}\,,%\label{def:flux-v2}
   	\quad \flux{\bfv}_{h,q}(\bfv_h) \coloneqq  
   	\begin{cases}
   		\avg{\bfv_h} &\text{on $\Gamma_h^{i}$}\\
   		\bfv^* &\text{on $\partial\Omega$}
   	\end{cases}\,,\\[-0.25mm]
   	\label{def:flux-q}
   	\flux{q}(q_h) &\coloneqq 
   	q_h 
   	\quad \text{on $\Gamma_h$}\,,\\[-0.25mm]
   	\label{def:flux-S}
   	\flux{\bfS}(\bfv_h, \bfS_h,\bfL_h) &\coloneqq 
   	\avg{\bfS_h} \hspace*{-0.1em}- \hspace*{-0.1em}\alpha\, \SSS_{\smash{\sss}}\big(h^{-1}\jump{(\bfv_h     \hspace*{-0.1em}-\hspace*{-0.1em} \bfv_0^*)\hspace*{-0.1em}\otimes\hspace*{-0.1em} \bfn}\big)
   	\quad \text{on $\Gamma_h$}\,,\\[-0.25mm]
   	\label{def:flux-K}
   	\flux{\bfK}(\bfv_h) &\coloneqq 
   	\avg{\bfK_h} 
   	\quad \text{on $\Gamma_h$}\,,\\[-0.25mm]
   	\label{def:flux-F}
   	\flux{\bfG}(\PiDG\bfG) &\coloneqq  
   	\bigavg{\PiDG \bfG} \quad\text{on $\Gamma_h$}\,,
   \end{align}
   where the operator $\SSS_{\smash{\sss}}$ is defined as in
   \eqref{eq:flux}.  Thus we arrive~at~an~inf-sup stable system
   without using a pressure stabilization. {It is also possible
     to work with a
   discontinuous pressure. In this case one has to modify the fluxes as follows:
   $\flux{q}(q_h) \coloneqq  \avg{q_h} $~on~$\Gamma_h$~and
   $\flux{\bfv}_{h,q}(\bfv_h) \coloneqq 
   \avg{\bfv_h}+h\jump{q_h\bfn}$ on $\Gamma_h^{i}$,
   $\flux{\bfv}_{h,q}(\bfv_h) \coloneqq  \bfv^* $ on
   $\partial\Omega$.}

    As in Part I of the paper (cf.~\cite{kr-pnse-ldg-1}), we arrive at the {flux
      formulation} of~\eqref{eq:p-navier-stokes}, which reads:
    For given $\bfg\in \smash{L^{p'}(\Omega)} $, find
    $(\bfL_h,\bfS_h,\bfK_h,\bfv_h,q_h)^\top \in \Xhk \times
    \Xhk\times\Xhk\times\Vhk \times \Qhkco$ such that for all
    $(\bfX_h,\bfY_h,\bfZ_h,\bfz_h,z_h)^\top$
    $ \in \Xhk \times \Xhk\times\Xhk\times\Vhk \times \Qhkc$, it holds
    \begin{align}
        \hskp{\bfL_h}{\bfX_h} &= \bighskp{\Ghk \bfv_h }{
          \bfX_h}\,,  \notag
        \\
        \hskp{\bfS_h}{\bfY_h} &= \hskp{\SSS(\bfL_h^{\textup{sym}})}{
          \bfY_h}\,, \notag
           \\
       \label{eq:DG} \hskp{\bfK_h}{\bfZ_h} &= \hskp{\bfv_h\otimes \bfv_h}{
          \bfZ_h}\,, 
        \\[-0.5mm]
        \bighskp{\bfS_h-\tfrac{1}{2}\bfK_h-q_h\mathbf{I}_d}{\Dhk \bfz_h} &=
        \hskp{\bfg-\tfrac{1}{2}\bfL_h\,\bfv_h}{\bfz_h} \notag
        \\[-0.5mm]
        &\quad - \alpha \,\bigskp{\SSS_{\smash{\avg{\abs{\Pia\bfL_h^{\textup{sym}}}}}}(h^{-1} \jump{\bfv_h \otimes
            \bfn})}{ \jump{\bfz_h \otimes \bfn}}_{\Gamma_h}\,, \notag\\[-1.5mm]
        \bighskp{\Divhk \bfv_h}{z_h}&=0\,.  \notag
    \end{align}  
    
    Now, we eliminate in the system~\eqref{eq:DG} the
    variables $ \bfL_h\hspace{-0.15em}\in\hspace{-0.15em} \Xhk$, $ \bfS_h\hspace{-0.15em}\in\hspace{-0.15em}     \Xhk$~and~${\bfK_h\hspace{-0.15em}\in\hspace{-0.15em} \Xhk}$ to derive a system only expressed in
    terms of the two variables ${\bfv_h\in\Vhk}$~and~${q_h\in
      \Qhkco}$. {To this end, we observe that it follows from \eqref{eq:DG}$_{1,2,3}$ that
      $\bfL_h= \Ghk \bfv_h + \Rhk\bfv^*$, $\bfL_h^{\textup{sym}}=\Dhk\bfv_h$,  $\bfS_h=\PiDG\SSS(\bfL_h^{\textup{sym}})$, $\bfK_h= \PiDG
      (\bfv_h \otimes \bfv_h)$.
    If we insert this into \eqref{eq:DG}$_4$,
    we get the discrete counterpart~of~Problem~(Q):}
    
    \textbf{Problem (Q$_h$).} For given $\bfg\in L^{p'}(\Omega)$, find $(\bfv_h,q_h)^\top\!\in\! \Vhk\times \Qhkco$ such~that~for~all $ (\bfz_h,z_h)^\top \in \Vhk\times\Qhkco$, it holds
    \begin{align}
        \bighskp{\SSS(\Dhk \bfv_h)-\tfrac{1}{2}\bfv_h\otimes \bfv_h-q_h\mathbf{I}_d}{\Dhk
          \bfz_h}\label{eq:primal1}
        &=     \bighskp{\bfg-\tfrac{1}{2}[\Ghk \bfv_h]\bfv_h}{\bfz_h}
        \\[-0.5mm]
        &\quad - \alpha \big\langle\SSS_{\smash{\sssl}}(h^{-1} \jump{\bfv_h\otimes\notag
            \bfn}), \jump{\bfz_h \otimes \bfn}\big\rangle_{\Gamma_h}\!,\\[-1.5mm]
      \bighskp{\Divhk \bfv_h}{z_h}&=0 \notag\,.
    \end{align}   
    Next, we eliminate in the system~\eqref{eq:primal1}, the
    variable $q_h\!\in\! \Qhkco$ to derive~a~system~only~expressed in terms of the single variable $\bfv_h\in \Vhk$. To this end, we introduce the space
    \begin{align*}
        V_h^k(0)\coloneqq \big\{\bfv_h\in V_h^k\mid \bighskp{\Divhk \bfv_h}{z_h}=0\textup{ for all }z_h\in     \Qhkc\big\}\,.
    \end{align*}
    Consequently, since $\hskp{z_h\mathbf{I}_d}{\Dhk
          \bfz_h}=\hskp{z_h}{\Divhk
          \bfz_h}=0$ for all $z_h\in \Qhkc$ and $\bfz_h\in V_h^k(0)$,
    we~get~the~discrete counterpart of Problem (P):
    
    \textbf{Problem (P$_h$).} For given $\bfg\!\in\! L^{p'}(\Omega)$, find $\bfv_h\!\in\! V_h^k(0)$ such that~for~all~${\bfz_h\!\in\! V_h^k(0)}$, it holds
    \begin{align} %\label{eq:primal2}
        %\begin{aligned}
          \bighskp{\SSS(\Dhk \bfv_h)-\tfrac{1}{2}\bfv_h\otimes \bfv_h}{\Dhk
          \bfz_h}&=     \bighskp{\bfg-\tfrac{1}{2}[\Ghk \bfv_h]\bfv_h}{\bfz_h}\label{eq:primal2}\\[-0.5mm]
       &\quad - \alpha \big\langle\SSS_{\smash{\sssl}}(h^{-1} \jump{\bfv_h\otimes
            \bfn}), \jump{\bfz_h \otimes \bfn}\big\rangle_{\Gamma_h}\,.\notag
        %\end{aligned}
    \end{align}
    Problem (Q$_h$) and Problem (P$_h$) are called {primal formulations} of
    the system \eqref{eq:p-navier-stokes}. 

    Well-posedness (i.e., solvability), stability (i.e., a priori
    estimates), and (weak) convergence of Problem (Q$_h$) and Problem
    (P$_h$) are proved in Part~I~of the paper~(cf.~\cite{kr-pnse-ldg-1}).
    
\section{Convergence rates}\label{sec:rates}

Let us start with the main result of this paper:
\begin{theorem}
  \label{thm:error}
  Let $\SSS$ satisfy Assumption~\ref{assum:extra_stress} with
  $p\in(2,\infty)$ and $\delta> 0$,~let~$k\in \mathbb{N}$, and let
  $\bfg\in L^{p'}(\Omega)$. Moreover, let
  $(\bfv,q)^\top \in \Vo(0)\times \Qo$ be a solution of
  Problem (Q) (cf.~\eqref{eq:q1}, \eqref{eq:q2}) with
  $\bfF(\bfD \bfv) \in W^{1,2}(\Omega)$%and $q\in  W^{1,p'}(\Omega)$ %\footnote{Note that $W^{1,\phi^*}(\Omega)=W^{1,p'}(\Omega)$.},
  and let $(\bfv_h,q_h)^\top \in \Vhk(0)\times \Qhkco$ be a solution
  of Problem (Q$_h$) (cf.~\eqref{eq:primal1}) for $\alpha>0$.   Then,
  there exists a constant $c_0 >0$, depending only on the characteristics of
  $\SSS$, $\delta^{-1}>0$, the chunkiness $\omega_0>0$, 
  $\alpha^{-1}>0$, and $k\in \mathbb{N}$, such that if $\norm{\nabla
    \bfv}_2\le c_0$, then, it holds
  \begin{align*}
    \bignorm{\smash{\bfF\big(\Dhk \bfv_h\big)}\!-\! \bfF(\bfD
      \bfv)}_2^2 \!+\!
      m_{\phi_{\smash{\sssl}},h } (\bfv_h\!-\!\bfv)\leq c\, h^2 \norm{
    \bfF(\bfD \bfv) }_{1,2}^2\!+\!c\,\rho_{(\phi_{\abs{\bfD \bfv}})^*,\Omega}(h\nabla q)
  \end{align*}
  with a constant $c\hspace{-0.1em}>\hspace{-0.1em}0$ depending only on the characteristics of
  $\SSS$, $\delta^{-1}\hspace{-0.1em}>\hspace{-0.1em}0$,~the~\mbox{chunkiness} $\omega_0>0$, 
  $\alpha^{-1}>0$,  $k\in \mathbb{N}$, $\norm{\bfv
  }_{\infty}\ge 0$,  and $c_0>0$.
\end{theorem}

\begin{corollary}\label{cor:error}
  Let the assumptions of Theorem \ref{thm:error} be satisfied. Then,~it~holds
  \begin{align*}
    \bignorm{\smash{\bfF\big(\Dhk \bfv_h\big)} - \bfF(\bfD
      \bfv)}_2^2 +
      \,m_{\phi_{\smash{\sssl}},h } (\bfv_h-\bfv)\leq  c\, h^2 \norm{\bfF(\bfD \bfv) }_{1,2}^2+c\,h^{p'}\!\rho_{\phi^*,\Omega}(\nabla q)
  \end{align*}
  with a constant $c>0$ depending only on the characteristics of
  $\SSS$, $\delta^{-1}>0$,~the~chunkiness $\omega_0>0$,
  $\alpha^{-1}>0$, $k\in \mathbb{N}$, $\norm{\bfv }_{\infty}\ge 0$,
  and $c_0>0$. If, in addition, $\bfg\in L^{2}(\Omega)$,~then
  \begin{align*}
    &\bignorm{\smash{\bfF\big(\Dhk \bfv_h\big) - \bfF(\bfD
      \bfv)}}_2^2 +m_{\phi_{\smash{\sssl}},h } (\bfv_h-\bfv)
    \\[1.5mm]
    &\leq  c\, h^2 \smash{\big ( \norm{\bfF(\bfD \bfv)
      }_{1,2}^2+\big\|(\delta+\abs{\bfD\bfv})^{\smash{\frac{2-p}{2}}}\nabla q\big\|_2^2
      %\int_\Omega(\delta+\abs{\bfD\bfv})^{2-p}\abs{\nabla q}^2\, \textrm{d}x
      \big )
      }
  \end{align*}
  with a constant $c>0$ depending only on the characteristics of
  $\SSS$, $\delta^{-1}>0$,~the~chunkiness~${\omega_0>0}$, 
  $\alpha^{-1}>0$,  $k\in \mathbb{N}$, $\norm{\bfv
  }_{\infty}\ge 0$,  and $c_0>0$.
\end{corollary}

\begin{remark}
(i)  Note that in Theorem \ref{thm:error}, the assumption $\bfg\in L^{p'}(\Omega)$ is equivalent to $\nabla q\in L^{p'}(\Omega)$ and in Corollary \ref{cor:error},  the assumption $\bfg\in L^2(\Omega)$ can be replaced by $(\delta+\abs{\bfD\bfv})^{2-p}\vert \nabla q\vert^2\in L^1(\Omega)$.

{(ii) To show that $\bfF(\bfD \bfv) \in
W^{1,2}(\Omega)$ holds under reasonable assumptions is still an open problem
(cf.~\cite{hugo-petr-rose} for partial results). However, this
regularity is natural for elliptic problems of $p$-Laplace typ
(cf.~\cite{gia-mod-86,giu1}) and proved in the two-dimensional case
(cf.~\cite{KMS2}) and in any dimension in the space periodic setting,
since it follows from interior regularity (cf.~\cite{hugo-petr-rose}).}
\end{remark}

\begin{corollary}
  \label{cor:error_F*}
 Let the assumptions of Theorem \ref{thm:error} be satisfied.~Then,~it~holds
  \begin{align*}
      \norm{\smash{\bfF^*(\PiDG\SSS(\Dhk\bfv_h))-\bfF^*(\SSS(\bfD \bfv))}}_2^2 \le c\,\norm{\smash{\bfF(\Dhk\bfv_h) - \bfF(\bfD \bfv)}}_2^2 +c\,  h^2\, \norm{\nabla \bfF(\bfD \bfv) }_2^2
  \end{align*}
  with a constant $c>0$ depending only on the characteristics of
  $\SSS$, the~chunkiness~${\omega_0>0}$,  and  $k\in \mathbb{N}$.
\end{corollary}

%\marginpar{replace ass. on g by ones for q}%

In order to prove these results, we need to derive a system similar~to~\eqref{eq:primal1},~which~is satisfied by a solution
of~our~\mbox{original}~\mbox{problem}~\eqref{eq:p-navier-stokes}. Moreover,
to treat the terms coming from the extra stress tensor, which depends on the symmetric
part~of~the~gradient, we need to generalize several auxiliary results,
which are proved in \cite{kr-phi-ldg} for the case of a full gradient,
to the case of symmetric gradients. This is possible due to the local
character of the $L^2$-projection and the local 
Korn inequality \cite[Theorem~6.13]{john}.

Let us first find a system similar to
\eqref{eq:primal1}, which is satisfied by a solution
of~\mbox{problem}~\eqref{eq:p-navier-stokes}.
Using the notation ${\bfL=\nabla\bfv}$,
${\bfS =\SSS(\bfL^{\textup{sym}})}$, ${\bfK =\bfv\otimes \bfv}$,
we~find~that
$(\bfv, \bfL, \bfS,\bfK)^\top \in W^{1,p}(\Omega)\times
L^p(\Omega)\times L^{p'}(\Omega)\times L^{p'}(\Omega)$. If, in
addition,~$\bfS,\bfK, q \in
W^{1,1}(\Omega)$, we observe as in \cite{dkrt-ldg}, i.e., using
integration-by-parts, the
projection~properties~of~$\PiDG$, the definition of the discrete
gradient, and of the jump~functional,~that
\begin{align}\label{eq:cont}
  \begin{aligned}
    \hskp{\bfL}{\bfX_h} &= \hskp{\nabla\bfv}{ \bfX_h}\,,
    \\
    \hskp{\bfS}{\bfY_h} &= \bighskp{\SSS(\bfL^{\textup{sym}})}{ \bfY_h}\,,
    \\
    \hskp{\bfK}{\bfZ_h} &= \bighskp{\bfv\otimes \bfv}{ \bfZ_h}\,,
    \\
    \bighskp{\bfS-\tfrac{1}{2}\bfK-q\mathbf{I}_d}{\Dhk \bfz_h} &= \hskp{\bfg-\tfrac{1}{2}\bfL\bfv}{\bfz_h} +
      \bigskp{\avg{\bfS}-\bigavg{\PiDG\bfS}}{\jump{\bfz_h\otimes
          \bfn}}_{\Gamma_h}\\&\quad+
          \tfrac{1}{2}\bigskp{\bigavg{\PiDG\bfK}-\avg{\bfK}}{\jump{\bfz_h\otimes
          \bfn}}_{\Gamma_h}
          \\&\quad+
          \bigskp{\bigavg{\PiDG(q\mathbf{I}_d)}-\avg{q\mathbf{I}_d}}{\jump{\bfz_h\otimes
          \bfn}}_{\Gamma_h}\,,
    \\
    \hskp{\Divhk\bfv}{\bfz_h} &= 0
            \end{aligned}
\end{align}  
is satisfied for all $(\bfX_h,\bfY_h,\bfZ_h,\bfz_h)^\top \in \Xhk \times \Xhk\times \Xhk \times \Vhk$.
As a result, using \eqref{eq:cont}, \eqref{eq:primal1} and
\eqref{eq:div-dg}, we~arrive~at
\begin{align}
     &\bighskp{\SSS(\Dhk \bfv_h) - \SSS(\bfD
      \bfv)}{\Dhk \bfz_h}
    + \alpha \big\langle\SSS_{\smash{\sssl}}(h^{-1} \jump{\bfv_h\otimes
        \bfn}), \jump{\bfz_h \otimes \bfn}\big\rangle_{\Gamma_h}\label{eq:errorprimal}\\
    &=\hskp{z_h-q}{\Divhk \bfz_h}+b_h(\bfv,\bfv,\bfz_h)-b_h(\bfv_h,\bfv_h,\bfz_h)\notag
    +\bigskp{\avg{\bfS}-\bigavg{\PiDG\bfS}}{\jump{\bfz_h\otimes
          \bfn}}_{\Gamma_h}
          \\&\quad+
          \tfrac{1}{2}\bigskp{\bigavg{\PiDG\bfK}-\avg{\bfK}}{\jump{\bfz_h\otimes
          \bfn}}_{\Gamma_h}\notag +
          \bigskp{\bigavg{\PiDG(q\mathbf{I}_d)}-\avg{q\mathbf{I}_d}}{\jump{\bfz_h\otimes
          \bfn}}_{\Gamma_h}\,,\notag
\end{align}
which is satisfied for all $(\bfz_h,z_h)^\top \in \Vhk(0)\times
\Qhkc$. Here, we denoted~the~discrete~convective~term by
$b_h\colon\WDG\times\WDG\times\WDG \to  \mathbb{R}$, which  is defined via
\begin{align*}
    b_h(\bfx_h,\bfy_h,\bfz_h)\coloneqq \tfrac{1}{2}\hskp{\bfz_h\otimes \bfx_h}{\Ghk\bfy_h}-\tfrac{1}{2}\hskp{\bfy_h\otimes \bfx_h}{\Ghk\bfz_h}
\end{align*}
for all $(\bfx_h,\bfy_h,\bfz_h)^\top\in \WDG\times\WDG\times\WDG$.

Next, we derive the results needed to treat the extra stress tensor
depending only on the symmetric part of the gradient.

\begin{lemma}\label{lem:poin_F}
        Let $\SSS$ satisfy Assumption~\ref{assum:extra_stress} with
        $p\!\in\! (1,\infty)$ and $\delta\!\ge \!0$. \!Then,~there~exists \hspace{-0.2mm}a \hspace{-0.2mm}constant \hspace{-0.2mm}$c\!>\!0$, \hspace{-0.2mm}depending \hspace{-0.2mm}only \hspace{-0.2mm}on \hspace{-0.2mm}the \hspace{-0.2mm}characteristics \hspace{-0.2mm}of \hspace{-0.2mm}$\SSS$ \hspace{-0.2mm}and~\hspace{-0.2mm}the~\hspace{-0.2mm}\mbox{chunkiness}~\hspace{-0.2mm}${\omega_0\!>\!0}$, such that for every  $\bfw_h \in
      W^{1,p}(\mathcal{T}_h)$ with $\bfF(\bfD_h\bfw_h)\in W^{1,2}(\mathcal{T}_h)$ and $K \in \mathcal{T}_h$,~it~holds
        \begin{align}
           \int_K{
        \bigabs{\bfF(\bfD_h \bfw_h) - \bfF\big(\mean{\bfD_h \bfw_h}_K\big)}^2 \,\mathrm{d}x}&\leq c\,h^2 \int_K{\abs{\nabla\bfF(\bfD_h\bfw_h)}^2\,\mathrm{d}x}\,,\label{lem:poin_F.1}\\
         \bignorm{\bfF(\bfD_h\bfw_h)-\bfF\big(\Pia\bfD_h\bfw_h\big)}_2^2&\leq c\,h^2 \,\bignorm{\nabla_h\bfF(\bfD_h\bfw_h)}_2^2\,.\label{lem:poin_F.2}
        \end{align}
        In particular, \eqref{lem:poin_F.1} also applies if we replace
        $K$ with $S_\gamma\coloneqq \bigcup\{K\in \mathcal{T}_h\mid
        \gamma\in \partial K\}$ or $\smash{S_K=\bigcup\{K'\in \mathcal{T}_h\mid \partial K'\cap \partial K\}}$.
    \end{lemma}
    
\begin{proof}
      Follows from \cite[Lemma A.3]{bdr-phi-stokes} and Poincar\'e's inequality.
\end{proof}

\begin{proposition}\label{prop:app_V}
    Let $\SSS$ satisfy Assumption~\ref{assum:extra_stress} with $p\in (1,\infty)$ and $\delta\ge 0$,~and~let $k\in \mathbb{N}$. Then, there exists a constant $c>0$, depending only on the characteristics~of~$\SSS$, the chunkinesss $\omega_0>0$, and $k\in \mathbb{N}$, such that for every $K \in \mathcal{T}_h$ and  ${\bfw_h\in
      W^{1,p}(\mathcal{T}_h)}$, it~holds
      \begin{align}
        \int_K \bigabs{\bfF (\bfD_h \bfw_h) - \bfF \big(\bfD_h \PiDG \bfw_h\big)}^2
        \,\mathrm{d}x &\leq c\,  \int_K
        \bigabs{\bfF(\bfD_h \bfw_h) - \bfF\big(\mean{\bfD_h \bfw_h}_K\big)}^2 \,\mathrm{d}x\,.\label{prop:app_V.0.1}
      \end{align}
      In addition, there exists a constant $c>0$, depending only on
      the characteristics of $\SSS$, the chunkinesss $\omega_0>0$, and
      $k\in \mathbb{N}$ such that for every $K \in \mathcal{T}_h$  and
      ${\bfw_h \in
      W^{1,p}(\mathcal{T}_h)}$ with ${\bfF(\bfD_h\bfw_h)\in W^{1,2}(\mathcal{T}_h)}$, it holds
      \begin{align}
        \int_K \bigabs{\bfF (\bfD_h \bfw_h) - \smash{\bfF \big(\bfD_h \PiDG \bfw_h\big)}}^2
        \,\mathrm{d}x &\leq c\,  h\,\int_K
        \bigabs{\nabla\bfF(\bfD_h \bfw_h)}^2 \,\mathrm{d}x\,,\label{prop:app_V.0.2}\\
          \bignorm{\bfF (\bfD_h \bfw_h) - \bfF \big(\bfD_h \PiDG \bfw_h\big)}^2_2&\leq c\,h^2\,
        \norm{\nabla_h\bfF(\bfD_h \bfw_h)}^2_2\,.\label{prop:app_V.0.3}
      \end{align}
    \end{proposition}
  
\begin{proof}
      This result is essentially proved in \cite[Thm.~5.7]{dr-interpol}
      and \cite[Thm.~3.4]{bdr-phi-stokes},~if~one~notices that the
      assumption on the interpolation operators therein are satisfied with~$S_K$ replaced by $K$. For the convenience of the reader, we
      carry~out~these~arguments~here. For every
      $\bfw_h \in W^{1,p}(\mathcal{T}_h)$, 
      Lemma~\ref{lem:hammer} implies that
      $\bfF(\bfD_h \bfw_h) \in L^2(\Omega)$. We fix an arbitrary
      $K \in \mathcal{T}_h$ and choose a linear function
      $\mathfrak{p}_h\colon\Omega\to      \mathbb{R}^d $ such that
            ${\nabla \mathfrak{p}_h =
        \mean {\nabla
          \bfw_h}_K}$~in~$K$,
      which~implies that $ \bfD \mathfrak{p}_h = \mean{\bfD \bfw_h}_K$~in~$K$. Using Lemma~\ref{lem:hammer},~the~triangle~inequality,
      and $\PiDG\mathfrak{p}_h=\mathfrak{p}_h$ in $\Vhk$ since already
      $\mathfrak{p}_h\in \Vhk$, we find that
    \begin{align}\label{prop:app_V.1}
            \int_K &\bigabs{\bfF (\bfD \bfw_h) - \bfF\big(\bfD \PiDG \bfw_h\big)}^2 \,\mathrm{d}x
            \leq c\, \int_K \phi_{|\bfD \bfw_h|}\big(\bigabs{\bfD \bfw_h - \bfD
              \PiDG\bfw_h}\big) \,\mathrm{d}x
            \\
            &\leq c\, \int_K \phi_{|\bfD \bfw_h|}(\abs{\bfD \bfw_h - \bfD \mathfrak{p}_h})
            \,\mathrm{d}x+ c\, \int_K \phi_{|\bfD \bfw_h|}\big(\bigabs{ \bfD\PiDG (\bfw_h-\mathfrak{p}_h
              )}\big) \,\mathrm{d}x 
              \eqqcolon I_1 + I_2\,.\notag
    \end{align}
    To apply the (local) Orlicz-stability of $\PiDG$ (cf.~\cite[Lemma A.5]{kr-phi-ldg})~to~$I_2$ we use the shift change~from~Lemma~\ref{lem:shift-change},~i.e.,~we~bound~$I_2$~as~follows
    \begin{align}\label{prop:app_V.2}
      I_2\leq c \int_K\phi_{|\bfD \mathfrak{p}_h|}\big(\bigabs{ \bfD\PiDG (\bfw_h-\mathfrak{p}_h
        )}\big) \,\mathrm{d}x+ c\,I_1\,.
    \end{align}
    Then, we can resort to the (local) Orlicz-stability of $\PiDG$
    (cf.~\cite[Lemma A.5]{kr-phi-ldg}), to Korn's
    inequality in the Orlicz setting (cf.~\cite[Theorem 6.13]{john}),
    also  using that  $\nabla \mathfrak{p}_h  =\mean{\nabla \bfw_h}_K$ in $K$,
    and use another shift change from Lemma~\ref{lem:shift-change} 
    to obtain
    \begin{align}\label{prop:app_V.3}
      \int_K\phi_{|\bfD \mathfrak{p}_h|}\big(\bigabs{ \bfD\PiDG (\bfw_h-\mathfrak{p}_h
        )}\big) \,\mathrm{d}x 
     \leq c\, 
     \int_K \phi_{|\bfD \mathfrak{p}_h|}(\abs{ \nabla
        \bfw_h-\nabla\mathfrak{p}_h }) \,\mathrm{d}x
      \leq c\,I_1\,.
    \end{align}
    Combining \hspace{-0.2mm}\eqref{prop:app_V.1}--\eqref{prop:app_V.3}, \hspace{-0.2mm}also  \hspace{-0.2mm}using  \hspace{-0.2mm}$\bfD \mathfrak{p}_h \!=\!\mean{\bfD \bfw_h}_K$~\hspace{-0.2mm}in~\hspace{-0.2mm}$K$ \hspace{-0.2mm}and \hspace{-0.2mm}\eqref{eq:hammera}, \hspace{-0.2mm}we \hspace{-0.2mm}conclude~\hspace{-0.2mm}that~\hspace{-0.2mm}\eqref{prop:app_V.0.1} applies.
    Then, \eqref{prop:app_V.0.2} follows~from Lemma \ref{lem:poin_F} and \eqref{prop:app_V.0.3} follows from \eqref{prop:app_V.0.2} via summation with respect to $K\in \mathcal{T}_h$.
    \end{proof}

Based on the previous results, we  generalize \cite[Lemma
5.5]{kr-phi-ldg} to the~case~of~symmetric
gradients. 

\begin{lemma}\label{lem:e7}
  Let $\SSS$ satisfy Assumption~\ref{assum:extra_stress}  with $p\in
  (1,\infty)$~and~${\delta\ge 0}$,~and~${j\in \setN_0} $. Then, there exists a constant $c>0$, depending only on the
    characteristics of ${\SSS}$,~the~chun-kiness $ \omega_0>0$, and $j\in \setN_0$ such that for every  $\bfX \in L^p(\Omega)$, $\bfY \in L^{p'}(\Omega)$,~${\bfw_h \in W^{1,p}(\mathcal{T}_h)}$ with
   $\bfF(\bfD_h \bfw_h) \in W^{1,2}(\mathcal{T}_h)$, it~holds
  \begin{align*}
      \bignorm{\bfF(\bfD_h \bfw_h) -\smash{\bfF\big(\Uppi^j_h\bfX\big)}}_2^2 &\le
      c\, h^2 \norm{\nabla_h \bfF(\bfD_h \bfw_h) }_2^2 + c\,
      \norm{\bfF(\bfD_h \bfw_h) -\bfF(\bfX)}_2^2\,,\\
        \bignorm{\bfF^*(\SSS(\bfD_h \bfw_h)) -\smash{\bfF^*\big(\Uppi^j_h\SSS(\bfY)\big)}}_2^2 &\le
      c\, h^2 \norm{\nabla_h \bfF(\bfD_h \bfw_h) }_2^2\\&\quad + c\,
      \norm{\bfF^*(\SSS(\bfD_h \bfw_h)) -\bfF^*(\SSS(\bfY))}_2^2\,.
  \end{align*}
\end{lemma}

\begin{proof}\let\qed\relax
  Resorting to Lemma \ref{lem:poin_F} and Lemma
  \ref{lem:hammer}, we find that
  \begin{align}\label{eq:e7.1}
  \begin{aligned}
    &\bignorm{\bfF(\bfD_h \bfw_h) -\smash{\bfF\big(\Uppi^j_h\bfX\big)}}_2^2
    \\
    &\le 2\,\bignorm{\bfF(\bfD_h \bfw_h) -\smash{\bfF\big(\Pia \bfD_h\bfw_h\big)}}_2^2 + 2\,
      \bignorm{\bfF\big(\Pia \bfD_h\bfw_h\big) -\smash{\bfF\big(\Uppi^j_h\bfX\big)}}_2^2
    \\
    &\le c\, h^2 \norm{\nabla_h \bfF(\bfD_h \bfw_h) }_2^2 + c\,
      \smash{\rho_{\phi_{\smash{\abs{\Pia \bfD_h\bfw_h}}},\Omega}\big( \Pia \bfD_h \bfw_h
      - \Uppi^j_h \bfX\big)}\,.
      \end{aligned}
  \end{align}
  The identity  $\Uppi^j_h\Pia \bfD_h \bfw_h=\Pia \bfD_h \bfw_h$, the Orlicz-stability of $\Uppi^j_h$
    (cf.~\cite[Lemma A.5]{kr-phi-ldg}),  Lemma \ref{lem:hammer}, and Lemma \ref{lem:poin_F}  yield
  \begin{align}\label{eq:e7.2}
  \begin{aligned}
    &\rho_{\phi_{\smash{\abs{\Pia \bfD\bfv}}},\Omega}\big( \Pia \bfD \bfv
    - \Uppi^j_h \bfX\big)
    =       \rho_{\phi_{\smash{\abs{\Pia \bfD_h\bfw_h}}},\Omega}\big( \Uppi^j_h\big(\Pia \bfD_h \bfw_h
      - \bfX\big)\big)
    \\
    &\le c\,      \rho_{\phi_{\smash{\abs{\Pia \bfD_h\bfw_h}}},\Omega}\big( \Pia \bfD_h \bfw_h
      - \bfX\big)
    \\
    &\le c\, \bignorm{\smash{\bfF\big(\Pia\bfD_h\bfw_h\big)} -\bfF(\bfD_h\bfw_h) }_2^2  +c\,
      \norm{ \bfF(\bfD_h \bfw_h) -\bfF(\bfX) }_2^2  
    \\
    &\le c\, h^2 \norm{\nabla_h \bfF(\bfD_h \bfw_h) }_2^2 +c\, \norm{
      \bfF(\bfD_h \bfw_h) -\bfF(\bfX) }_2^2 \,.
      \end{aligned}
  \end{align}
  Combining \eqref{eq:e7.1} and \eqref{eq:e7.2}, we conclude the first assertion. The second assertion follows analogously, if we use \eqref{eq:F-F*3} and \cite[Lemma 4.4]{dkrt-ldg}.
\end{proof}

Next, we treat the jump operator and several expressions defined on faces.

\begin{proposition}\label{prop:e4}
      Let $\SSS$ satisfy Assumption~\ref{assum:extra_stress} with
      $p\in (1,\infty)$ and $\delta \ge 0$,~and~let $k\in \mathbb{N}$. Then, there exists a constant $c>0$, depending only on the
    characteristics~of~${\SSS }$, the chunkiness $\omega_0\!>\!0$, and $k\!\in\! \mathbb{N}$ such that for every $\bfv \!\in\! W^{1,p}(\Omega)$ with~${\bfF(\bfD \bfv) \!\in\! W^{1,2}(\Omega)}$ and $\gamma\in \Gamma_h$,~it~holds
       \begin{align}
        \label{eq:e4.1}
        \int_{S_\gamma}{\phi_{\abs{{\bfD
          \bfv}}} \big(\bigabs{ \boldsymbol{\mathcal R}^k_{h,\gamma}
          (\bfv-\PiDG \bfv )}\big)\, \mathrm{d}x}&\leq c\,h^2\, \int_{S_\gamma}{\vert \nabla\bfF(\bfD \bfv)\vert^2 \,\mathrm{d}x}\,,\\
          \label{eq:e4.2}
        \rho_{\phi_{\abs{\bfD \bfv}},\Omega}\big({ \Rhk(\bfv-\PiDG\bfv)}\big)
        &
        \le  c\, h^2\,\norm{\nabla \bfF(\bfD  \bfv)}_2^2\,.  
      \end{align}
    \end{proposition}
    
\begin{proof}
      Due to
      $\smash{\Rhk\bfv \hspace{-0.05em}=\hspace{-0.05em}\sum_{\gamma \in
          \Gamma_h} \!\boldsymbol{\mathcal R}^k_{h,\gamma}\bfv}$,
      $\smash{\textup{supp}\boldsymbol{\mathcal
          R}^k_{h,\gamma}\bfv\hspace{-0.05em}\subseteq\hspace{-0.05em}
        S_\gamma}$,
      $\gamma\hspace{-0.05em}\in\hspace{-0.05em} \Gamma_h$,
      $\smash{\Omega \hspace{-0.05em}=\hspace{-0.05em}\bigcup_{\gamma\in
          \Gamma_h}{\!S_\gamma}}$, and
      that~for~each~${\gamma\in\Gamma_h}$, $S_\gamma$ consist of at
      most two elements, it suffices to~prove~\eqref{eq:e4.1}~to~get
      \eqref{eq:e4.2} via summation with respect to ${\gamma\in\Gamma_h}$.  The shift change in Lemma
      \ref{lem:shift-change}, the (local) stability properties of
      $\smash{\boldsymbol{\mathcal
          R}^k_{h,\gamma}}$~in~\cite[Lemma~A.1]{kr-phi-ldg}, Lemma
      \ref{lem:poin_F} for $S_\gamma$,
      the approximation property of
      $\PiDG$ in \cite[Corollary~A.19]{kr-phi-ldg} for
      $\bfv-\smash{\PiDG} \bfv$, yield %\allowdisplaybreaks
      \begin{align}
       \label{eq:e4.3}     
       \begin{aligned}
        &\int_{S_\gamma}\phi_{\abs{{\bfD
          \bfv}}} \big(\bigabs{ \boldsymbol{\mathcal R}^k_{h,\gamma}
          (\bfv-\PiDG \bfv )}\big)\, \textrm{d}x
        \\
        &\le c\, \int_{S_\gamma}{\phi_{\abs{\mean{\bfD
          \bfv}_{S_\gamma}}} \big(\bigabs{ \boldsymbol{\mathcal R}^k_{h,\gamma}
          (\bfv-\PiDG \bfv)}\big)+\bigabs{\bfF(\bfD
          \bfv)-\bfF\big(\mean{\bfD\bfv}_{S_\gamma}\big)}^2\, \textrm{d}x}
        \\
        &\le c\,h\int_{\gamma}{\phi_{\abs{\mean{\bfD
          \bfv}_{S_\gamma}}} \big( h^{-1} |\jump{(\bfv-\PiDG\bfv)\otimes
      \bfn}|\big )\,
          \textrm{d}s}+ c\, h^2\,\int_{S_\gamma}{\abs{\nabla \bfF(\bfD\bfv)}^2   \, \textrm{d}x}
        \\
        &\le c\,\int_{S_\gamma}{ \phi_{\abs{\mean{\bfD
          \bfv}_{S_\gamma}}} \big(|{ \nabla_h(\bfv-\PiDG\bfv)}|\big)\, \textrm{d}x}+
          c\, h^2\,\int_{S_\gamma}{\abs{\nabla \bfF(\bfD\bfv)}^2   \, \textrm{d}x}\,.
          \end{aligned}
      \end{align}
      To treat the first term in the last line of \eqref{eq:e4.3}, we
      proceed analogously to the proof of Proposition
      \ref{prop:app_V}, i.e., we choose 
      $\mathfrak{p}_h \in V_h^1 $ with
      ${\nabla \mathfrak{p}_h = \mean {\nabla \bfv}_{S_\gamma}}$ in $S_\gamma$,
      use the (local) Orlicz-stability of $\PiDG$
    (cf.~\cite[Lemma A.5]{kr-phi-ldg}), Korn's
    inequality in the Orlicz setting (cf.~\cite[Theorem 6.13]{john}),
    Proposition \ref{lem:hammer}, and 
    Poincar\'e's inequality~on~$S_\gamma $ to find that
      \begin{align}\label{eq:e4.4}
        \begin{aligned}
          \int_{S_\gamma} \phi_{\abs{\mean{\bfD
         \bfv}_{S_\gamma}}} \big(|{ \nabla_h(\bfv-\PiDG\bfv)}|\big)   \, \textrm{d}x
          &\leq c\int_{S_\gamma}\bigabs{\bfF(\bfD
          \bfv)-\bfF\big(\mean{\bfD\bfv}_{S_\gamma}\big)}^2\, \textrm{d}x
          \\&
          \leq 
          c\, h^2\,\int_{S_\gamma}
         \abs{\nabla \bfF(\bfD \bfv)}^2   \, \textrm{d}x\,.
          \end{aligned}
      \end{align}
      Combining \eqref{eq:e4.3} and \eqref{eq:e4.4}, we conclude that \eqref{eq:e4.1} applies. Then, \eqref{eq:e4.2}~follows~from  \eqref{eq:e4.1}
      via summation with respect to $\gamma \in \Gamma_h$.
    \end{proof}
    
\begin{corollary}
        \label{cor:app_V}
      Let $\SSS$ satisfy Assumption~\ref{assum:extra_stress} with
      $p\in (1,\infty)$ and $\delta\ge 0$,~and~let $k\in \mathbb{N}$. Then, there exists a constant $c>0$, depending only on the
    characteristics~of~${\SSS}$, the chunkiness $\omega_0\!>\!0$, and $k\!\in \!\mathbb{N}$ such that for every $\bfv \!\in\! W^{1,p}_0(\Omega)$~with~${\bfF(\bfD\bfv) \!\in\! W^{1,2}(\Omega)}$ and $\gamma\in \Gamma_h$, it holds
       \begin{align}
        \label{cor:app_V.1}
        \int_{S_\gamma} \bigabs{\bfF (\bfD\bfv) - \smash{\bfF \big(\Dhk \PiDG \bfv\big)}}^2
        \,\mathrm{d}x &\leq c\,  h^2\int_{S_\gamma}
        \bigabs{\nabla\bfF(\bfD\bfv)}^2 \,\mathrm{d}x\,,\\\label{cor:app_V.2}
          \bignorm{\bfF (\bfD\bfv) - \bfF \big(\Dhk \PiDG \bfv\big)}^2_2&\leq c\,h^2\,
        \norm{\nabla\bfF(\bfD\bfv)}^2_2\,.
      \end{align}
     \end{corollary}
     
\begin{proof}
       Using that $\Dhk \PiDG \bfv=\bfD_h\PiDG
       \bfv+(\Rhk\PiDG\bfv)^{\textup {sym}}$ and $\Rhk\bfv
       =\mathbf{0}$ in $L^p(\Omega)$, \eqref{eq:hammera}, \eqref{prop:app_V.0.2} and \eqref{eq:e4.1}, we find that
       \begin{align}\label{cor:app_V.3}
            \begin{aligned}
                \int_{S_\gamma} \bigabs{\bfF (\bfD\bfv) - \bfF \big(\Dhk \PiDG \bfv\big)}^2
                \,\mathrm{d}x
                &\leq
                  c\int_{S_\gamma}{ \bigabs{\bfF (\bfD\bfv) - \bfF \big(\bfD_h\PiDG\bfv\big)}^2\,\textrm{d}x}
                  \\&\quad+c\int_{S_\gamma}{\phi_{\abs{{\bfD
          \bfv}}} \big(\bigabs{ \boldsymbol{\mathcal R}^k_{h,\gamma}
          (\bfv-\PiDG \bfv )}\big)\, \textrm{d}x}
      \\
      &\leq c\,  h^2\int_{S_\gamma}
                \abs{\nabla\bfF(\bfD\bfv)}^2 \,\mathrm{d}x\,.
            \end{aligned}    
       \end{align}
       Then, we conclude that \eqref{cor:app_V.2} applies via summation of \eqref{cor:app_V.3} with respect to $\gamma\in \Gamma_h$.
     \end{proof}

\begin{proposition}\label{prop:PiDGapproxmspecial}
	Let $\SSS$ satisfy Assumption~\ref{assum:extra_stress} with
        $p\in (1,\infty)$ and $\delta\ge 0$, and let $k \in \setN_0$. Then, there exists a constant $c>0$, depending only on the
    characteristics of ${\SSS}$, the chunkiness $\omega_0\!>\!0$, and
    $k \!\in\! \setN_0$ such that for every ${\bfv\!\in\!
      W^{1,p}(\Omega)}$ with ${\bfF(\bfD\bfv) \!\in\! W^{1,2}(\Omega)}$ and $\bfw_h\in W^{1,p}(\mathcal{T}_h)$, it holds
	\begin{align}\label{prop:PiDGapproxmspeciallocal}
	&\begin{aligned}
	&h\, \int_\gamma
        \phi_{\abs{\bfD\bfv}}\big(h^{-1}\abs{\jump{(\bfw_h - \PiDG
            \bfw_h)\otimes\bfn}}\big)\,\textup{d}s \\[-.5mm]
        &\leq c\, \int_{S_\gamma}{\phi_{\abs{\bfD\bfv}}\big(\vert     \nabla_h\bfw_h\vert\big)\,\textup{d}x}+c\,h^2\,\int_{S_\gamma}{\abs{\nabla\bfF(\bfD\bfv)}^2\,\textup{d}x} \,, 
	\end{aligned}\\&
%	\end{align}
% \begin{align}
          \label{prop:PiDGapproxmspeciallocal1}
          \begin{aligned}
        &h\, \int_\gamma   \phi_{\abs{\bfD\bfv}}\big(\abs{\avg{\bfw_h} - \avg{\PiDG
            \bfw_h}}\big)\,\textup{d}s\\[-.5mm]
        &\leq c\, \int_{S_\gamma}{\phi_{\abs{\bfD\bfv}}\big(h\,\vert
          \nabla_h\bfw_h\vert\big)\,\textup{d}x}+c\,h^2\,\int_{S_\gamma}{\abs{\nabla\bfF(\bfD\bfv)}^2\,\textup{d}x}
        \,, \\[-6mm]
        \end{aligned}
	\end{align}
	\begin{align}
          m_{\phi_{\abs{\bfD\bfv}},h}\big(\bfw_h - \PiDG\bfw_h\big)
          &\leq c\,\rho_{\phi_{\abs{\bfD\bfv}},\Omega}(\nabla_h\bfw_h)
            +c\,h^2\,\|\nabla\bfF(\bfD\bfv)\|_2^2\,.\label{prop:PiDGapproxmspecialglobal}
        \end{align}
        All assertions remain valid for $\bfw_h\in W^{1,p'}(\mathcal{T}_h)$ if we replace
        $\phi_{\abs{\bfD\bfv}}$ by $(\phi_{\abs{\bfD\bfv}})^*$.
\end{proposition}

\begin{proof}
	Using the convexity of $\phi_{\abs{\bfD\bfv}}$, twice a shift change from Lemma \ref{lem:shift-change}, the local trace inequalities \cite[(A.20)]{kr-phi-ldg} and Lemma \ref{lem:poin_F}, for~every~${\gamma\in \Gamma_h}$, 
	we~find~that
	\begin{align}\label{prop:PiDGapproxmspecial.1}
	\begin{aligned}
		&h \int_\gamma \phi_{\abs{\bfD\bfv}}\big(h^{-1}\abs{\jump{(\bfw_h - \PiDG \bfw_h)\otimes\bfn}}\big)\,\textup{d}s \\&\quad\leq h\sum_{K\in \mathcal{T}_h;K\subseteq S_\gamma}{\!\int_\gamma{\phi_{\abs{\mean{\bfD\bfv}_K}}\big(h^{-1} \bigabs{ \bfw_h \!-\! (\PiDG \bfw_h)|_K}\big)\!+\!\bigabs{\bfF(\bfD\bfv) \!-\!\bfF\big(\mean{\bfD\bfv}_K\big)}^2\,\textup{d}s}}
		\\&
		\quad\leq c\,\sum_{K\in \mathcal{T}_h;K\subseteq S_\gamma}{\int_K{\phi_{\abs{\mean{\bfD\bfv}_K}}(\vert \nabla\bfw_h\vert)+h^2\,\abs{\nabla\bfF(\bfD\bfv)}^2\,\textup{d}x}}
		\\&\quad\leq \int_{S_\gamma}{\phi_{\abs{\bfD\bfv}}(\vert \nabla_h\bfw_h\vert)\,\textup{d}x}+h^2\int_{S_\gamma}{\abs{\nabla\bfF(\bfD\bfv)}^2\,\textup{d}x}\,,
			\end{aligned}\hspace{-5mm}
	\end{align}
	i.e., \eqref{prop:PiDGapproxmspeciallocal}.  Then,
        \eqref{prop:PiDGapproxmspecialglobal} follows from
        \eqref{prop:PiDGapproxmspeciallocal} via summation with
        respect to $\gamma\in
        \Gamma_h$. The estimate \eqref{prop:PiDGapproxmspeciallocal1} is proved
        analogously to \eqref{prop:PiDGapproxmspeciallocal}.
\end{proof}
%\pagebreak
\begin{corollary}
 \label{cor:PiDGapproxmspecial}
	Let $\SSS$ satisfy Assumption~\ref{assum:extra_stress} with
        $p\in (1,\infty)$ and $\delta\ge 0$, and let $k \in \setN$. Then, there exists a constant $c>0$, depending only on the
    characteristics~of~${\SSS}$, the chunkiness $\omega_0\!>\!0$, and $k\! \in\! \setN$ such that for every $\bfv\!\in\! W^{1,p}(\Omega) $~with~${\bfF(\bfD\bfv)\! \in\! W^{1,2}(\Omega)}$, it holds
	\begin{align}
		\label{cor:PiDGapproxmspeciallocal}
	    h \int_\gamma \phi_{\abs{\bfD\bfv}}\big(h^{-1}\abs{\jump{(\bfv - \PiDG     \bfv)\otimes\bfn}}\big)\,\textup{d}s&\leq c\, h^2\int_{S_\gamma}{\abs{\nabla\bfF(\bfD\bfv)}^2\,\textup{d}s}    \,,\\
		m_{\phi_{\abs{\bfD\bfv}},h}\big(\bfv - \PiDG \bfv\big) &\leq c\,h^2\,\|\nabla\bfF(\bfD\bfv)\|_2^2\,.\label{cor:PiDGapproxmspecialglobal}
	\end{align}
\end{corollary}

\begin{proof}
  We use Proposition \ref{prop:PiDGapproxmspecial} with
  $\bfw_h\coloneqq \bfv - \PiDG \bfv\in \WDG$ and treat the
  resulting term
  $\smash{\int_{S_\gamma}{\!\phi_{\abs{\bfD\bfv}}(\vert \nabla_h(\bfv - \PiDG
    \bfv)\vert)\,\textup{d}x}}$ with a shift change in Lemma
  \ref{lem:shift-change}~and~\eqref{eq:e4.4}. This proves
  \eqref{cor:PiDGapproxmspeciallocal}. Then,
  \eqref{cor:PiDGapproxmspecialglobal} follows from
  \eqref{cor:PiDGapproxmspeciallocal} via summation~with respect to ${\gamma\in \Gamma_h}$.
\end{proof}

\begin{lemma}\label{lem:e5}
	Let $\SSS$ satisfy Assumption~\ref{assum:extra_stress}  with
        $p\in (1,\infty)$ and $\delta\ge 0$. Then, there exists a constant $c>0$, depending only on the
    characteristics of ${\SSS}$ and the chunkiness $\omega_0\hspace{-0.075em}>\hspace{-0.075em}0$, such that for every $\bfX\hspace{-0.075em}\in\hspace{-0.075em} L^p(\Omega)$ and $\bfv \hspace{-0.075em}\in\hspace{-0.075em} W^{1,p}(\Omega)$~with~${\bfF(\bfD \bfv)\hspace{-0.075em}\in \hspace{-0.075em} W^{1,2}(\Omega)}$,~it~holds
	\begin{align}
		\label{lem:e5.1}
		h\,\rho_{\phi_{\abs{\bfD\bfv}},\Gamma_h}\big(\vert \bfD\bfv\vert -\avg{\abs{\Pia \bfX}}\big)
		&\le  c\, h^2\, \norm{\nabla \bfF(\bfD\bfv)}_2^2  +c\,\norm{\bfF(\bfD \bfv)-\bfF(\bfX)}^2_2\,.
	\end{align}
\end{lemma}

\begin{proof}
	Using that $\vert \vert \bfD\bfv\vert -\avg{\abs{\Pia \bfX}}\vert\leq\avg{\vert\bfD\bfv -\Pia \bfX\vert}$ in $\Gamma_h$, 
	the convexity of $\phi_{\abs{\bfD\bfv}}$, \eqref{eq:hammera} and the trace inequality \cite[(A.17)]{kr-phi-ldg}, for every $\gamma\in \Gamma_h$, we find that
	\begin{align}\label{lem:e5.2}
	    \begin{aligned}
	        &h\int_{\gamma}{\phi_{\abs{\bfD\bfv}}\big(\bigabs{\vert \bfD\bfv\vert -\avg{\abs{\Pia     \bfX}}}\big)\,\mathrm{d}s}\leq c\,h\int_{\gamma}{\phi_{\abs{\bfD\bfv}}\big(\avg{\vert\bfD\bfv -\Pia     \bfX\vert}\big)\,\mathrm{d}s}\\&\quad\leq 
		     c\,h\sum_{K\in \mathcal{T}_h;K\subseteq S_\gamma}{\int_{\gamma}{\phi_{\abs{\bfD\bfv}}\big(\bigabs{\bfD\bfv    -(\Pia \bfX)|_K}\big)	\,\mathrm{d}s}}
		     \\&\quad\leq
		     c\,h\sum_{K\in \mathcal{T}_h;K\subseteq S_\gamma}{\int_{\gamma}{\bigabs{\bfF(\bfD\bfv) -\bfF\big((\Pia    \bfX)|_K\big)}^2	\,\mathrm{d}s}} 
		     \\&\quad\leq
		     c\,h^2 \,\int_{S_\gamma}{\abs{\nabla\bfF(\bfD\bfv)}^2	   \,\mathrm{d}x}+c\,\int_{S_\gamma}{\bigabs{\bfF(\bfD\bfv) -\bfF\big(\Pia \bfX\big)}^2	\,\mathrm{d}x}\,. 
		 \end{aligned}
	\end{align}
	Then, the assertion follows via summation of \eqref{lem:e5.2} with respect to $\gamma\in \Gamma_h$~and~resorting to Lemma~\ref{lem:e7} with $j=0$.
\end{proof}
    
Now we have prepared everything to prove our main result Theorem \ref{thm:error}.    
\begin{proof}[Proof of Theorem \ref{thm:error}]
  Setting $\bfe_h\coloneqq \bfv_h-\bfv\in \WDG$ and  resorting~to
  ${\SSS_a(\bfA)\colon \!\bfA}$ $\sim\varphi_a(\vert \bfA^{\textup{sym}}\vert)$ uniformly in $\bfA\in \mathbb{R}^{d\times d}$ and $a\ge 0$, which follows from \eqref{eq:hammera} with $\bfB=\mathbf{0}$ (cf.~Remark \ref{rem:sa}), as well as
  $\bfe_h = \PiDG \bfe_h + \PiDG \bfv -\bfv$ and
  $\jump{\bfv\otimes \bfn}=\bfzero$ on $\Gamma_h$, we find that
  \begin{align}\label{thm:error.1}
        \begin{aligned}
              &c\,\bignorm{\smash{\bfF\big(\Dhk \bfv_h\big)} - \bfF(\bfD
                \bfv)}_2^2+c\,\alpha\,m_{\phi_{\smash{\sssl}},h } (\bfe_h)\\&\quad\leq \bighskp{\SSS\big(\Dhk \bfv_h\big) - \SSS(\bfD
              \bfv)}{\Dhk \bfe_h}
            \\
            &\quad\quad+ \alpha \big\langle\SSS_{\smash{\sssl}}\big(h^{-1} \jump{\bfe_h\otimes
                \bfn}\big), \jump{\bfe_h \otimes \bfn}\big\rangle_{\Gamma_h}
              \\
              &\quad= \bighskp{\SSS\big(\Dhk \bfv_h\big) - \SSS(\bfD
              \bfv)}{\Dhk \PiDG\bfe_h}
            \\&\quad\quad+ \alpha \big\langle\SSS_{\smash{\sssl}}\big(h^{-1} \jump{\bfv_h\otimes
                \bfn}\big), \jump{\PiDG\bfe_h \otimes \bfn}\big\rangle_{\Gamma_h}
                \\&\quad\quad+ \bighskp{\SSS\big(\Dhk \bfv_h\big) - \SSS(\bfD
              \bfv)}{\Dhk (\PiDG\bfv-\bfv)}
            \\&\quad\quad+ \alpha \big\langle\SSS_{\smash{\sssl}}\big(h^{-1} \jump{\bfe_h\otimes
                \bfn}\big), \jump{(\PiDG\bfv-\bfv) \otimes \bfn}\big\rangle_{\Gamma_h}
                \\&\quad=\vcentcolon I_1+\alpha I_2 + I_3 + \alpha I_4\,.
        \end{aligned}
  \end{align}
  Using $\Dhk\bfv =\bfD\bfv$, the $\varepsilon$-Young inequality \eqref{ineq:young} with $\psi=\varphi_{\vert \bfD\bfv\vert}$, two equivalences in \eqref{eq:hammera},  and
  Corollary~\ref{cor:app_V}~yield
  \begin{align}\label{thm:error.2}
        \begin{aligned}
            I_3&\leq \varepsilon\,\bignorm{\smash{\bfF\big(\Dhk \bfv_h\big)}-\bfF(\bfD
            \bfv)}_2^2+c_\varepsilon\, \bignorm{\bfF(\bfD\bfv)-\smash{\bfF\big(\Dhk\PiDG\bfv\big)}}_2^2\\&\leq
            \varepsilon\,\bignorm{\smash{\bfF\big(\Dhk \bfv_h\big)}-\bfF(\bfD
            \bfv)}_2^2+c_\varepsilon \,h^2\,\norm{\nabla\bfF(\bfD\bfv)}_2^2\,.
        \end{aligned}
  \end{align}
  Using the $\varepsilon$-Young inequality \eqref{ineq:young}
  with ${\psi\!=\!\phi_{\smash{\sssl}}}$, that uniformly in $\bfA\!\in\!
  \mathbb{R}^{d\times d}$ and $a\!\ge\! 0$ there holds $(\varphi_a)^*(\abs{\SSS_a(\bfA)})\! \sim\!\varphi_a(\vert \bfA^{\textup{sym}}\vert)$, which follows from \eqref{eq:hammera} with ${\bfB=\mathbf{0}}$ (cf.~Remark \ref{rem:sa}),  a shift change in Lemma
  \ref{lem:shift-change}, Lemma~\ref{lem:e5}~and~Corollary~\ref{cor:PiDGapproxmspecial}~yield
  \begin{align}\label{thm:error.3}
        \begin{aligned}
            I_4&\leq \varepsilon \,m_{\phi_{\smash{\sssl}},h } (\bfe_h)+c_\varepsilon\,m_{\phi_{\smash{\sssl}},h } \big(\bfv-\PiDG\bfv\big)
            \\&\leq \varepsilon \,m_{\phi_{\smash{\sssl}},h } (\bfe_h)+c_\varepsilon\,c_\kappa\,m_{\phi_{\abs{\bfD\bfv}},h } \big(\bfv-\PiDG\bfv\big)       \\&\quad+c_\varepsilon\,\kappa\,h\,\rho_{\phi_{\abs{\bfD\bfv}},\Gamma_h}\big({ \abs{\bfD\bfv}-\sssl} \big)
            \\&\leq \varepsilon \,m_{\phi_{\smash{\sssl}},h } (\bfe_h)+c_\varepsilon\,(c_\kappa+\kappa)\,h^2\,\norm{\nabla\bfF(\bfD\bfv)}^2_2       \\&\quad+c_\varepsilon\,\kappa\,\bignorm{\smash{\bfF\big(\Dhk \bfv_h\big)}-\bfF(\bfD
            \bfv)}_2^2
            \,.
        \end{aligned}
  \end{align}
In view of our assumption Lemma \ref{lem:pres} yields  that
\eqref{eq:errorprimal} is applicable. Thus, % Appealing to the error equation
% \eqref{eq:errorprimal},
for every $z_h\in \Qhkc$, we have that
    \begin{align}\label{thm:error.4}
        \begin{aligned}
            I_1+\alpha I_2&=\bighskp{z_h-q}{\Divhk \PiDG\bfe_h}\\&\quad+b_h\big(\bfv,\bfv,\PiDG\bfe_h\big)-b_h\big(\bfv_h,\bfv_h,\PiDG\bfe_h\big)
            \\&\quad+\bigskp{\avg{\bfS}-\bigavg{\PiDG\bfS}}{\jump{\PiDG\bfe_h\otimes
            \bfn}}_{\Gamma_h}\\&\quad+
            \tfrac{1}{2}\bigskp{\bigavg{\PiDG\bfK}-\avg{\bfK}}{\jump{\PiDG\bfe_h\otimes
            \bfn}}_{\Gamma_h}
            \\&\quad+
            \bigskp{\bigavg{\PiDG(q\mathbf{I}_d)}-\avg{q\mathbf{I}_d}}{\jump{\PiDG\bfe_h\otimes
            \bfn}}_{\Gamma_h}
            \\&=\vcentcolon J_1+\dots + J_5 \,.
          \end{aligned}
    \end{align}
    So, let us next estimate $J_1,\dots,J_5$:
    
    (\textit{$J_1$}): Using $\PiDG \bfe_h=\bfv -\PiDG\bfv +
    (\bfv_h-\bfv)$, $\Dhk \bfv =\bfD\bfv$, the $\varepsilon$-Young inequality~\eqref{ineq:young} for 
  $\psi =\phi_{\abs{\bfD \bfv}}$, \eqref{eq:hammera} and Corollary~\ref{cor:app_V}, for every $z_h\in \Qhkc$, we find that
    \begin{align}\label{thm:error.5}
        \begin{aligned}
            J_1&=\bighskp{(z_h-q)\mathbf{I}_d}{\bfD\bfv -\Dhk \PiDG\bfv}-\bighskp{(z_h-q)\mathbf{I}_d}{\bfD\bfv-\Dhk \bfv_h}
            \\&\leq c\,\rho_{(\varphi_{\abs{\bfD\bfv}})^*,\Omega}({q-z_h})+c\,\bignorm{\bfF(\bfD\bfv)-\smash{\bfF\big(\Dhk\PiDG\bfv\big)}}_2^2\\&\quad+c_\varepsilon\,\,\rho    _{(\varphi_{\abs{\bfD\bfv}})^*,\Omega}({q-z_h})+
            \varepsilon\, \bignorm{\bfF(\bfD\bfv)-\smash{\bfF\big(\Dhk\bfv_h\big)}}_2^2
            \\&\leq c_\varepsilon\,\rho_{(\varphi_{\abs{\bfD\bfv}})^*,\Omega}({q-z_h})+c\,h^2\,\norm{\nabla\bfF(\bfD\bfv)}_2^2+\varepsilon\,
            \bignorm{\bfF(\bfD\bfv)-\smash{\bfF\big(\Dhk\bfv_h\big)}}_2^2\,.\hspace*{-4mm}
        \end{aligned}
    \end{align}
    Next,  denote by $\Pi_h^{\smash{Q},k}:L^{p'}(\Omega)\to \Qhkc$, the Clem\'ent quasi-interpolation operator,~cf.~\cite{BF1991}. Then, 
    using twice a shift change in Lemma \ref{lem:shift-change},
    \cite[Lemma 5.2]{bdr-phi-stokes} with
    $\psi=(\phi_{\abs{\mean{\bfD\bfv}_{S_K}}})^*$, where
    $\smash{S_K=\bigcup\{K'\!\in\! \mathcal{T}_h\mid \partial
      K'\!\cap\! \partial K\}}$, and Lemma \ref{lem:poin_F} for $S_K$,
    we obtain
    \begin{align}\label{thm:error.6}
    \begin{aligned}
      \rho_{(\varphi_{\abs{\bfD\bfv}})^*,\Omega}\big({q-\Pi_h^{\smash{Q},k}
      q}\big)
      &\leq c\, \rho_{(\varphi_{\abs{\bfD\bfv}})^*,\Omega}(h\,{\nabla q})+c\,h^2\,\norm{\nabla\bfF(\bfD\bfv)}_2^2\,.
        \end{aligned}
    \end{align}
    {We also used that $(\varphi_a)^*(t) \le \varphi^*(t)$ for all
    $t,a\ge0$, valid~for~${p>2}$  (cf.~\cite{bdr-phi-stokes}), which
    together with Lemma \ref{lem:pres} yields that the modular in
    \eqref{thm:error.6} is finite. }
    Thus, choosing $z_h =\Pi_h^{\smash{Q},k} q\in \Qhkc$ in \eqref{thm:error.5} and resorting to \eqref{thm:error.6}, we deduce that
       \begin{align}\label{thm:error.5a}
        \begin{aligned}
            \abs{J_1}&\leq 
        c_\varepsilon\, \rho_{(\varphi_{\abs{\bfD\bfv}})^*,\Omega}(h\,{\nabla q})
      +c_\varepsilon\,h^2\,\norm{\nabla\bfF(\bfD\bfv)}_2^2
      +\varepsilon\,
            \bignorm{\bfF(\bfD\bfv)-\smash{\bfF\big(\Dhk\bfv_h\big)}}_2^2\,.
        \end{aligned}
    \end{align}
    
    (\textit{$J_2$}): By definition, we have that $b_h(\bfv_h,\PiDG\bfe_h,\PiDG\bfe_h)=0$, which yields 
    \begin{align}\label{thm:error.7}
        \begin{aligned}
            J_2&= b_h(\bfv,\bfv,\PiDG\bfe_h)\pm b_h(\bfv,\PiDG\bfv,\PiDG\bfe_h)\\&\quad-b_h(\bfv_h,\bfv_h,\PiDG\bfe_h)\pm     b_h(\bfv_h,\PiDG\bfv,\PiDG\bfe_h)
            \\
            &= b_h(\bfv,\bfv-\PiDG\bfv,\PiDG\bfe_h)+ b_h(\bfv,\PiDG\bfv,\PiDG\bfe_h)\\&\quad-b_h(\bfv_h,\PiDG\bfe_h,\PiDG\bfe_h)-     b_h(\bfv_h,\PiDG\bfv,\PiDG\bfe_h)
            \\
            &= b_h(\bfv,\bfv-\PiDG\bfv,\PiDG\bfe_h)- b_h(\bfe_h,\PiDG\bfv,\PiDG\bfe_h)
            %\\&
            =\vcentcolon J_{2,1}-J_{2,2}\,.
        \end{aligned}
    \end{align}
    In addition, we have that
    \begin{align}\label{thm:error.8}
        \begin{aligned}
            2 J_{2,1}\hspace{-0.1em}=\hspace{-0.1em}\bighskp{\PiDG\bfe_h\otimes \bfv}{\Ghk(\bfv\hspace{-0.1em}-\hspace{-0.1em}\PiDG\bfv)}\hspace{-0.1em}-\hspace{-0.1em}\bighskp{(\bfv-\PiDG\bfv)\otimes \bfv}{\Ghk\PiDG\bfe_h}\hspace{-0.1em}=\vcentcolon \hspace{-0.1em}    J_{2,1}^1\hspace{-0.1em}+\hspace{-0.1em}J_{2,1}^2\,.
        \end{aligned}
    \end{align}
    % A \hspace{-0.1mm}Korn \hspace{-0.1mm}type \hspace{-0.1mm}inequality
    % \hspace{-0.1mm}(cf.~\hspace{-0.1mm}\cite[\hspace{-0.1mm}Lemma \hspace{-0.1mm}6.3]{mnr3}), \hspace{-0.1mm}$p\!>\!2$, \hspace{-0.1mm}and~\hspace{-0.1mm}\cite[\hspace{-0.1mm}Lemma~\hspace{-0.1mm}3.8]{bdr-7-5}~\hspace{-0.1mm}yield~\hspace{-0.1mm}${\bfv\!\in\! W^{2,2}(\Omega)}$ since
    % \begin{align}
    %   \label{eq:embed}
    %    \norm{\nabla^2\bfv}^2_2\le c\,
    %   \norm{\nabla\bfD\bfv}^2_2\le c\,\delta^{2-p}
    %   \norm{\nabla\bfF(\bfD\bfv)}^2_2 \,.
    % \end{align}
    In the proof of Lemma \ref{lem:pres} we showed that $\bfv\in
    W^{2,2}(\Omega) \vnor L^\infty(\Omega)$.~Thus,  the
    $\varepsilon$-Young inequality \eqref{ineq:young}~with~$\smash{\psi=\frac{1}{2}\abs{\cdot}^2}$; the
    $L^2$-stability~of~$\PiDG$ (cf.~\cite[Corollary~A.8]{kr-phi-ldg}),
    the Poincar\'e inequality \cite[Lemma A.34]{kr-phi-ldg},
    \eqref{eq:eqiv0}; the approximation property of~$\PiDG$
    \mbox{(cf.~\cite[(A.14)]{kr-phi-ldg})}, and the Korn type inequality
    in Proposition~\ref{prop:kornii} for $p=2$ imply
    \begin{align}
            \abs{J_{2,1}^1}&\leq \varepsilon\,\norm{\bfv}_\infty\bignorm{\PiDG\bfe_h}_2^2+c_\varepsilon\,\norm{\bfv}_\infty\bignorm{\smash{\Ghk(\bfv-\PiDG\bfv)}}_2^2\notag
            \\&\label{thm:error.9}\leq
      \varepsilon\,c_{\bfv}\norm{\bfe_h}_{\nabla,2,h}^2+c_{\varepsilon,\bfv}\,
      \bignorm{\bfv-\smash{\PiDG\bfv}}_{\nabla,
      2,h}^2%+c_{\varepsilon,\bfv}\, \bignorm{\bfv-\PiDG\bfv}_{\nabla, 2,h}^2  %h^2\,\norm{\nabla^2\bfv}_2^2
      \\
        &\leq \varepsilon\,c_{\bfv}\norm{\bfe_h}_{\bfD,2,h}^2+c_{\varepsilon,\bfv}\,h^2\,\norm{\nabla^2\bfv}_2^2\,,\notag
    \end{align}
    where $c_{\bfv}\hspace{-0.15em}>\hspace{-0.15em}0$ depends crucially on $\norm{\bfv}_\infty\hspace{-0.15em}\ge\hspace{-0.15em} 0$ and   $c_{\varepsilon,\bfv}\hspace{-0.15em}>\hspace{-0.15em}0$~crucially~on~${\norm{\bfv}_\infty\hspace{-0.15em}\ge\hspace{-0.15em} 0}$~and~${\varepsilon\hspace{-0.1em}>\hspace{-0.1em}0}$. 
    
    Similarly, using that $\bfv\hspace{-0.1em}\in\hspace{-0.1em} L^\infty(\Omega)$, the
    $\varepsilon$-Young inequality \eqref{ineq:young}~with~$\psi\hspace{-0.1em}=\hspace{-0.1em}\frac{1}{2}\abs{\cdot}^2$,~\eqref{eq:eqiv0}, the
    approximation property of $\PiDG$ (cf.~\cite[Lemma
    A.5]{kr-phi-ldg}), the DG-norm stability~of~$\PiDG$
    (cf.~\cite[(A.18)]{dkrt-ldg}), and the~Korn~type inequality in Proposition~\ref{prop:kornii} for $p=2$, we~obtain
    \begin{align}\label{thm:error.10}
        \begin{aligned}
        \abs{J_{2,1}^2}&\leq \varepsilon\,\norm{\bfv}_\infty\bignorm{\smash{\Ghk\PiDG\bfe_h}}_2^2+c_\varepsilon\,\norm{\bfv}_\infty\bignorm{\smash{\bfv-\PiDG\bfv}}_2^2
        \\&\leq \varepsilon\,c_{\bfv}\,\bignorm{\smash{\PiDG\bfe_h}}_{\nabla,2,h}^2+c_{\varepsilon,\bfv}\,h^4\,\norm{\smash{\nabla^2\bfv}}_2^2
        \\&\leq \varepsilon\,c_{\bfv}\norm{\bfe_h}_{\nabla,2,h}^2+c_{\varepsilon,\bfv}\,h^2\,\norm{\nabla^2\bfv}_2^2
        \\&\leq \varepsilon\,c_{\bfv}\norm{\bfe_h}_{\bfD,2,h}^2+c_{\varepsilon,\bfv}\,h^2\, 
      \norm{\smash{\nabla^2\bfv}}^2_2 \,.
        \end{aligned}
    \end{align}
    
    Moreover, we have that
    \begin{align}\label{thm:error.11}
        2 J_{2,2}=\bighskp{\PiDG\bfe_h\otimes\bfe_h}{\Ghk\PiDG\bfv}-\bighskp{\PiDG\bfv\otimes \bfe_h}{\Ghk\PiDG\bfe_h}=\vcentcolon J_{2,2}^1+J_{2,2}^2\,.
    \end{align}
    Using Hölder's inequality, the identity
    $\bfe_h=\PiDG\bfe_h+(\PiDG\bfv- \bfv)$, the discrete Sobolev embedding theorem (cf.~\cite[\!Theorem~5.3]{EP12}),~\eqref{eq:eqiv0}, the approximation properties of $\PiDG$ (cf.~\cite[Proposition~A.2]{br-parabolic}),~the~\mbox{DG-norm}~stability~of $\PiDG$ (cf.~\cite[(A.18)]{dkrt-ldg}) and 
    the Korn type~inequality~in~Proposition~\ref{prop:kornii} for $p=2$, we obtain
    \begin{align}\label{thm:error.12}
        \begin{aligned}
       \abs{ J_{2,2}^1}&\leq \bignorm{\PiDG\bfe_h}_4\bignorm{\Ghk\PiDG\bfv}_2\big(\bignorm{\PiDG\bfe_h}_4+\bignorm{\bfv-\PiDG\bfv}_4\big)
        \\
        & \leq
        c\, \bignorm{\PiDG\bfe_h}_{\nabla,2,h}\bignorm{\PiDG\bfv}_{\nabla,2,h}\big(\bignorm{\PiDG\bfe_h}_{\nabla,2,h}+h^{\frac
          54}\,\norm{\nabla^2\bfv}_2\big)
         \\&
        \leq
        c\,\norm{\bfe_h}_{\nabla,2,h}\norm{\nabla\bfv}_2\big(\norm{\bfe_h}_{\nabla,2,h}+h^{\frac
        54}\,\norm{\nabla^2\bfv}_2\big)
     %   \\&
     %   \leq
     %   c\,\smash{\big(\norm{\bfe_h}_{\bfD,2,h}+h\,\norm{\nabla^2\bfv}_2\big) \norm{\nabla\bfv}_2\big
     %   ( \norm{\bfe_h}_{\bfD,2,h}+h\,\norm{\nabla^2\bfv}_2\big)}
      \\
      &  \leq  c\,{\big(\norm{\bfe_h}_{\bfD,2,h}^2+h^2\,
      \norm{\nabla^2\bfv}^2_2 \big) \norm{\nabla\bfv}_2}\,,
        \end{aligned}
    \end{align}
    and 
    \begin{align}
    \label{thm:error.13}
        \begin{aligned}
        \abs{J_{2,2}^2}&\leq \bignorm{\PiDG\bfv}_4\bignorm{\Ghk\PiDG\bfe_h}_2\big(\bignorm{\PiDG\bfe_h}_4+\bignorm{\bfv-\PiDG\bfv}_4\big)
        \\&
        \leq c\, \bignorm{\PiDG\bfv}_{\nabla,2,h}\bignorm{\PiDG\bfe_h}_{\nabla,2,h}\big(\bignorm{\PiDG\bfe_h}_{\nabla,2,h}+h^{\frac
          54}\,\norm{\nabla^2\bfv}_2\big)
        \\&
        \leq  c\,{\big(\norm{\bfe_h}_{\bfD,2,h}^2+h^2\,
      \norm{\nabla^2\bfv}^2_2 \big) \norm{\nabla\bfv}_2}\,.
        \end{aligned}
    \end{align}
    Finally, combining \eqref{thm:error.9}, \eqref{thm:error.10} in
    \eqref{thm:error.8}, \eqref{thm:error.12},
    \eqref{thm:error.13} in \eqref{thm:error.11}, \eqref{thm:error.7}
    yields 
    \begin{align}\label{thm:error.13.2}
        \begin{aligned}
        \abs{J_2}\leq  \big(\varepsilon\,c_{\bfv}+c\,\norm{\nabla\bfv}_2\big)\norm{\bfe_h}_{\bfD,2,h}^2+\big(c_{\varepsilon,\bfv}+c\,\norm{\nabla\bfv}_2\big)\,h^2\,
      \norm{\nabla^2\bfv}^2_2 \,.
      \end{aligned}
    \end{align}
    
     (\textit{$J_3$}): We treat $J_3$ similarly as 
      $K_3$ in \cite{kr-phi-ldg} (cf.~\cite[proof of Theorem
     4.5]{dkrt-ldg}), i.e., using the $\varepsilon$-Young inequality \eqref{ineq:young} for
  $\psi =\phi_{\abs{\bfD \bfv}}$, we find that
  \begin{align}
    \label{eq:e9}
    \begin{aligned}
      \abs{J_3} &\le c_\varepsilon\,h\,\int_{\Gamma_h}
      (\phi_{\abs{\bfD \bfv}})^* \big (\big|\avg{\bfS}
      -\bigavg{\PiDG\bfS}\big|\big ) \, \mathrm{d}s +\varepsilon
      \,m_{\phi_{\abs{\bfD\bfv}},h}\big (\PiDG \bfe_h\big )
      \\
      &\eqqcolon c_\varepsilon\,\sum_{\gamma\in \Gamma_h} J^\gamma_{3,1}
      +\varepsilon\, J_{3,2}\,. 
    \end{aligned}
  \end{align}
  Using the identity $\PiDG\bfe_h=\bfe_h+(\bfv- \PiDG\bfv)$,
  a~shift~change~from~Lemma~\ref{lem:shift-change}, Lemma~\ref{lem:e5}
  and Corollary \ref{cor:PiDGapproxmspecial}, we find that
  \begin{align}
    \label{eq:e10}
    \begin{aligned}
      \abs{J_{3,2}}  &\leq c\, m_{\phi_{\abs{\bfD\bfv}},h} \big(\bfe_h\big)+c\, m_{\phi_{\abs{\bfD\bfv}},h} \big(\bfv-\PiDG \bfv\big)
      \\&\le   c\,m_{\phi_{\smash{\sss}},h} (\bfe_h) +c\,
    m_{\phi_{\abs{\bfD\bfv}},h} \big(\bfv-\PiDG \bfv\big)\\&\quad +c\, h
      \,\rho
    _{\phi_{\abs{\bfD\bfv}},\Gamma_h}\big({\abs{\bfD\bfv}-\sssl}\big)
      \\[-0.5mm]&\leq c\, m_{\phi_{\smash{\sss}},h} (\bfe_h)+c\, h^2 \norm{\nabla \bfF(\bfD\bfv) }_2^2 + c\,
      \|\bfF(\bfD\bfv) -\smash{\bfF\big(\Dhk\bfv_h\big)}\big\|_2^2\,.
    \end{aligned}
  \end{align}
  For each $\gamma \in \Gamma_h$, we choose some $K \in \mathcal T_h$
  such that $\gamma \subset \partial K$.  From $
  (\phi^*)_{\abs{\SSS(\bfA)}}(t) \sim (\phi_{\abs{\bfA}})^*(t)$, for
  all $t\ge
  0$, $\bfA \in \setR^{d\times d}$
  (cf.~\cite[(2.25)]{kr-phi-ldg}), \eqref{eq:hammerf};
  \eqref{eq:F-F*3}, again \eqref{eq:hammerf}; the trace inequalities
  \cite[Lemma~A.16)]{kr-phi-ldg}, the stability~of~$\PiDG$~in
  \cite[(A.12)]{kr-phi-ldg};  \eqref{eq:hammera},  and twice  Lemma \ref{lem:poin_F}, it~follows~that
%  \allowdisplaybreaks
  \begin{align}
    \abs{J_{3,1}^\gamma}
    &\le c\,h\bignorm{\bfF^*
      ({\SSS(\bfD\bfv)} ) -\smash{\bfF^*(\PiDG \SSS(\bfD\bfv))}}_{2,\gamma}^2 \label{eq:e11}
    \\
    &\le c\,h \bignorm{\bfF^*
      ({\SSS(\bfD\bfv)} )\! -\! \smash{\bfF^*( \SSS(\Pia\bfD\bfv))}}_{2,\gamma}^2
    \!\!+\!c\,h \bignorm{\smash{\bfF^*
      ( \SSS(\Pia\bfD\bfv))\!-\!
      \bfF^*( \PiDG \SSS(\bfD\bfv))}}_{2,\gamma}^2 \notag
    \\
    &\le  c\,h \bignorm{\smash{\bfF
      (\bfD\bfv )\! -\!\bfF(\Pia\bfD\bfv)}}_{2,\gamma}^2
        \!+\!c\,h\int _\gamma (\phi^*)_{\abs{\SSS(\Pia\bfD\bfv)}} \big (\PiDG (\SSS
      (\bfD\bfv)\!-\!\SSS(\Pia\bfD\bfv))\big ) \, \mathrm{d}s \notag
    \\
    &\le c\,\bignorm{\smash{\bfF(\bfD\bfv )-\bfF(\Pia\bfD\bfv
      )}}_{2,K}^2+c\,h^2\,\bignorm{\nabla\bfF(\bfD\bfv)}_{2,K}^2
      \notag 
    \\
    &\quad +c\,\int_K (\phi^*)_{\abs{\SSS(\Pia\bfD\bfv)}}\big (\SSS
      (\bfD\bfv)-\SSS(\Pia\bfD\bfv)\big ) \, \mathrm{d}x \notag
    \\
    &\le  c\,h^2\,\norm{\nabla \bfF(\bfD
      \bfv )}_{2,K}^2  +c\,\bignorm{\smash{\bfF(\bfD\bfv )-\bfF\big(\Pia\bfD\bfv \big)}}_{2,K}^2 \notag
     \\
    &\le c\, h^2\,\norm{\nabla \bfF(\bfD
      \bfv )}_{2,K}^2\,. \notag
  \end{align}
  Therefore, combining \eqref{eq:e10} and \eqref{eq:e11} in \eqref{eq:e9}, we conclude that 
   \begin{align}
    \label{eq:e9a}
    \abs{J_3}
    \leq c_\varepsilon\, h^2\,\norm{\nabla \bfF(\bfD
      \bfv )}_{2}^2 +\varepsilon \, c\, \big ( 
      \norm{\smash{\bfF(\bfD\bfv) \hspace{-0.1em} - \hspace{-0.1em}\bfF\big(\Dhk\bfv_h\big)}}_2^2  + m_{\phi_{\smash{\sss}},h} (\bfe_h) \big )\,.\hspace{-2mm}
  \end{align}
  
  (\textit{$J_4$}): 
    In view of $\bfv\in L^\infty(\Omega)\cap W^{1,2}_0(\Omega)$, we have that $\bfK =\bfv \otimes \bfv\in W^{1,2}_0(\Omega)$~with $\norm{\nabla     \bfK}_{2} \le c\, \norm{\bfv}_{\infty}\norm{\nabla \bfv}_2\leq c_{\bfv}\,\delta^{2-p}\|\bfF(\bfD\bfv)\|_{2} $. 
    Using~the~\mbox{$\varepsilon$-Young} inequality \eqref{ineq:young} with
    $\psi=\frac{1}{2}\vert\cdot\vert^2$; the approximation
    properties of $\PiDG$ (cf.~\cite[(A.24)]{kr-phi-ldg}), the DG-stability of $\PiDG$ (cf. \cite[(A.19)]{dkrt-ldg}) and the Korn type inequality in Proposition \ref{prop:kornii}, %the identity
    %$\PiDG\bfe_h=\bfe_h+(\bfv- \PiDG\bfv)$; the definition of
    %$\|\cdot \|_{2,\bfD,h}$, the
    %identity $\bfv-\PiDG\bfv=\bfv-\PiDG\bfv-\PiDG(\bfv-\PiDG\bfv)$, and the approximation
    %properties of $\PiDG$ (cf.~\cite[Corollary~A.19, Corollary A.8]{kr-phi-ldg}),
    we~get
    \begin{align}\label{thm:error.14}
        \hspace{-3mm}\begin{aligned}
        \abs{J_4}&\leq c_\varepsilon\,h\,\bignorm{\smash{\avg{\bfK} -\bigavg{\PiDG
            \bfK}}}^2_{2,\Gamma_h}+2\,\varepsilon \,h\,\bignorm{\smash{ h^{-1}\jump{\PiDG\bfe_h\otimes \bfn}}}^2_{2,\Gamma_h}
        \\[-0.25mm]
        &\leq c_\varepsilon\,h^2\,\bignorm{\nabla\bfK}^2_2
           +c\,\varepsilon \,\norm{\bfe_h}^2_{2,\nabla,h} %\!+2\,\varepsilon
           %\,h\,\bignorm{\smash{ h^{-1}\jump{(\bfv-\PiDG \bfv)\otimes
           %    \bfn}}}^2_{2,\Gamma_h}  
     % \\[-0.25mm]
      %     &\leq c_\varepsilon\,h^2\,\norm{\bfv}_{\infty}^2\norm{\nabla \bfv}_2^2
     % +c\,\varepsilon \,\norm{\bfe_h}^2_{2,\bfD,h}%  +\varepsilon\,c \,\bignorm{\smash{\nabla_h(\bfv-\PiDG\bfv)}}_2^2 
      \\
      &\leq c_{\varepsilon,\bfv}\,h^2\,\big(\norm{\nabla \bfv}_{2}^2+\|\nabla^2\bfv\|_2^2\big)
      +c\,\varepsilon \,\norm{\bfe_h}^2_{2,\bfD,h}\,.
      % \\
      %&\leq c_{\varepsilon,\bfv}\,h^2\,\delta^{2-p}
      %\norm{\bfF(\bfD\bfv)}^2_{1,2} 
      %+2\,\varepsilon \,\norm{\bfe_h}^2_{2,\bfD,h}\,.
      \end{aligned}\hspace{-2mm}
    \end{align}
    
    (\textit{$J_5$}): Using the $\varepsilon$-Young inequality
    \eqref{ineq:young} with $\psi=\phi_{\smash{\sss}}$, the identity
    $\PiDG\bfe_h=\bfe_h+(\bfv- \PiDG\bfv)$, twice a shift change from
    Lemma \ref{lem:shift-change}, the approximation properties of
    $\PiDG$ in Proposition \ref{prop:PiDGapproxmspecial} and Corollary
    \ref{cor:PiDGapproxmspecial}, Lemma \ref{lem:e5}, we obtain
    \begin{align}\label{thm:error.15}
        \begin{aligned}
        \abs{J_5}&\leq c_\varepsilon\,h\,\rho_{(\phi_{\smash{\sss}})^*,\Gamma_h} \big (\avg{q} -\bigavg{\PiDG q}\big )
      +\varepsilon \,m_{\phi_{\smash{\sss}},h}\big (\PiDG \bfe_h\big ) 
      \\
      &\leq c_\varepsilon \,c_\kappa\,h\,\rho_{(\phi_{\smash{\abs{\bfD\bfv}}})^*,\Gamma_h} \big (\avg{q}
      -\bigavg{\PiDG q}\big ) +c\,\varepsilon \,m_{\phi_{\smash{\sss}},h} (\bfe_h)
      \\
      &\quad +c\,\varepsilon
      \,m_{\phi_{\abs{\bfD\bfv}},h}\big (\bfv -\PiDG \bfv\big ) +(c_\varepsilon \kappa +c\, \varepsilon) \, h      \,\rho
    _{\phi_{\abs{\bfD\bfv}},\Gamma_h}\big({\abs{\bfD\bfv}-\sssl}\big)
      \\
      &\leq  c_\varepsilon \,c_\kappa\,\rho_{(\phi_{\smash{\abs{\bfD\bfv}}})^*,\Omega} ({h\,\nabla q} ) +c\,\varepsilon
     \,m_{\phi_{\smash{\sss}}} (\bfe_h)+c_\varepsilon \,c_\kappa\, h^2 \norm{\nabla
       \bfF(\bfD\bfv) }_2^2
     \\
     &\quad +(c_\varepsilon \kappa +c\, \varepsilon) \,\|\bfF\big(\Dhk \bfv_h\big) - \bfF(\bfD
        \bfv)\|_2^2\,.
      \end{aligned}
    \end{align}

    Collecting all the estimates
    \eqref{thm:error.1}--\eqref{thm:error.5}, \eqref{thm:error.5a},
    \eqref{thm:error.13.2}, \eqref{eq:e9a}, 
    \eqref{thm:error.14}, and \eqref{thm:error.15}, we proved
    \begin{align}\label{thm:error.16}
    \begin{aligned}
      &\bignorm{\smash{\bfF\big(\Dhk \bfv_h\big) - \bfF(\bfD
        \bfv)}}_2^2 + \alpha\,m_{\phi_{\smash{\sssl}},h } (\bfv_h-\bfv) %\\&\quad\leq I_3+\alpha I_4+ J_1+\dots+J_5
    \\    
    &\le (\varepsilon+(1+\alpha)\,\kappa\, c_\varepsilon)\,c\, \bignorm{\smash{\bfF(\bfD\bfv)
        -\bfF\big(\Dhk\bfv_h\big)}}_2^2
    +\varepsilon \,  c \,(1+\alpha)\, m_{\phi_{\smash{\sssl}},h } (\bfv_h-\bfv)\hspace*{-10mm}
    \\
     &\quad +c\, \big(\varepsilon\,c_{\bfv}+\norm{\nabla\bfv}_2\big)\norm{\bfe_h}_{\bfD,2,h}^2
    + c_\varepsilon \,(1+\alpha\,(c_\kappa+\kappa))\,h^2\, \norm{
        \nabla\bfF(\bfD \bfv) }_{2}^2
    \\
    &\quad+
        (c_{\varepsilon,\bfv}+\norm{\nabla\bfv}_2)\, h^2\,\big(\norm{\nabla\bfv}_2^2+\norm{\nabla^2\bfv}_2^2\big)
        +c_\varepsilon\, \rho_{(\phi_{\smash{\abs{\bfD\bfv}}})^*,\Omega} (h\,\nabla q)
      \\
      &\le \big(\varepsilon+(1+\alpha)\,\kappa\, c_\varepsilon+\delta^{2-p}(\varepsilon\,c_{\bfv}+\norm{\nabla\bfv}_2)\big)\,c\, \bignorm{\smash{\bfF(\bfD\bfv)
        -\bfF\big(\Dhk\bfv_h\big)}}_2^2 
    \\
        &\quad +\big(\varepsilon (1+\alpha )+\delta^{2-p}(\varepsilon\,c_{\bfv}+\norm{\nabla\bfv}_2)\big)\,  c \,m_{\phi_{\smash{\sssl}},h } (\bfv_h-\bfv)
      \\
      &\quad + \big(c_\varepsilon \,(1+c_\kappa+\alpha\,(c_\kappa+\kappa))+\delta^{2-p}(c_{\varepsilon,\bfv}+\norm{\nabla\bfv}_2)\big)\,h^2\, \norm{
        \bfF(\bfD \bfv) }_{1,2}^2
    \\
        &\quad+c_\varepsilon\,c_\kappa\, \rho_{(\phi_{\smash{\abs{\bfD\bfv}}})^*,\Omega} (h\,\nabla q)\,,
         \end{aligned}
    \end{align}
    where we  used for the last inequality sign \eqref{eq:embed1},
          %           that $\norm{\nabla\bfv}_2^2\!\leq\!
          %           \delta^{2-p}\norm{\bfF(\bfD\bfv)}_2^2$,~
    and~that 
    \begin{align}
        \norm{\bfe_h}^2_{\bfD,2,h}\leq c\,\delta^{2-p}\,\big(
      \norm{\smash{\bfF(\bfD\bfv)
      -\bfF\big(\Dhk\bfv_h\big)}}_2^2+m_{\phi_{\smash{\sssl}},h }
      (\bfv_h-\bfv)\big)\,,\label{eq:2toF}
    \end{align}
which follows from Proposition \ref{prop:equivalences} for $p=2$,
    Proposition \ref{lem:hammer}, and $p>2$.

   Choosing, for given $\alpha>0$, $\delta >0$, first $c_0>0$ sufficiently small, subsequently,~for this fixed $c_0>0$, $\varepsilon>0$ sufficiently
    small, and, then, for this fixed $\varepsilon>0$, 
    $\kappa>0$ sufficiently small, we can absorb the first and the second term on the right-hand side of the inequality chain \eqref{thm:error.16} in the left-hand side and conclude 
    the existence of a
    constant $c>0$, depending only on the characteristics of $\SSS$,
    the chunkiness $\omega_0>0$, $\delta^{-1}>0$,  $\alpha^{-1}<0$,~and~${\norm{\bfv}_{\infty}\ge 0}$, 
    such that
  \begin{align*}
    \bignorm{\smash{\bfF\big(\Dhk \bfv_h\big)}\!-\! \bfF(\bfD
      \bfv)}_2^2 \!+\!
      m_{\phi_{\smash{\sssl}},h } (\bfv_h\!-\!\bfv)\leq c\, h^2 \norm{
    \bfF(\bfD \bfv) }_{1,2}^2\!+\!c\,\rho_{(\phi_{\abs{\bfD \bfv}})^*,\Omega}(h\nabla q).\hspace*{-1mm}
  \end{align*}
    % Note that $W^{1,\phi^*}(\Omega)\simeq W^{1,p'}(\Omega)$, which
    % shows that the last modular is finite.
  This completes the proof of Theorem \ref{thm:error}.
\end{proof}

\begin{proof}[Proof of Corollary \ref{cor:error}]
    Using that $(\varphi_a)^*(h\,t)\hspace{-0.1em} \le \varphi^*(h\,t)\hspace{-0.1em} \le\hspace{-0.1em}
    c\, h^{p'} \varphi^*(t)$ for all $t,a\hspace{-0.1em}\ge\hspace{-0.1em} 0$, valid~for~${p\hspace{-0.1em}>\hspace{-0.1em}2}$ (cf.~\cite{bdr-phi-stokes}),  we deduce from Theorem \ref{thm:error} that
    \begin{align*}
       \norm{\bfF \smash{\big(\Dhk \bfv_h\big)} \hspace{-0.2em}-\hspace{-0.1em} \bfF(\bfD
        \bfv) }_{2}^{2} \hspace{-0.1em}+\hspace{-0.1em} \alpha\hspace{0.1em}m_{\phi_{\smash{\sssl}},h } (\bfv_h\hspace{-0.2em}-\hspace{-0.1em}\bfv)
      \hspace{-0.1em}\le\hspace{-0.1em} c\hspace{0.05em}h^2 \norm{
        \bfF(\bfD \bfv) }_{1,2}^2\hspace{-0.1em}+\hspace{-0.1em}c\hspace{0.05em}h^{p'}\!\! \rho_{\phi^*,\Omega}(\nabla q).
    \end{align*}
{  If, in addition, $\bfg \in L^2(\Omega)$ we proved in Lemma
  \ref{lem:pres} that $(\delta+ \abs{\bfD\bfv})^{2-p} \, \abs{\nabla q}^2 \in
  L^1(\Omega)$.  Moreover, there holds
  $(\varphi_a)^*(h\,t)\hspace{-0.1em} \sim \big( (\delta+a)^{p-1}
  +h\,t\big )^{p'-2}\, h^2\, t^2 \hspace{-0.1em}
  \le\hspace{-0.1em}(\delta+a)^{2-p}\, h^2\, t^2 $ for all
  $t,a\hspace{-0.1em}\ge\hspace{-0.1em} 0$, since $p>2$. Thus, we deduce from
  Theorem \ref{thm:error} that
  \begin{align*}
    &\bignorm{\bfF\smash{\big(\Dhk \bfv_h\big)} - \bfF(\bfD
      \bfv)}_2^2 +m_{\phi_{\smash{\sssl}},h } (\bfv_h-\bfv)
    \\
    &\leq  c\, h^2 \big ( \norm{\bfF(\bfD \bfv)
      }_{1,2}^2+\big\|(\delta+\abs{\bfD\bfv})^{\smash{\frac{2-p}{2}}}\nabla q\big\|_2^2
      %\int_\Omega(\delta+\abs{\bfD\bfv})^{2-p}\abs{\nabla q}^2\, \textrm{d}x
      \big )\,.
  \end{align*}}
\end{proof}

\begin{proof}[Proof of Corollary \ref{cor:error_F*}]
  Lemma \ref{lem:e7} with
  $\bfY=\Dhk\bfv_h$, $j=k$, and \eqref{eq:F-F*3} 
  imply that\enlargethispage{3mm}
  \begin{align*}
    \norm{\bfF^*\smash{(\PiDG \SSS(\Dhk\bfv_h))-\bfF^*(\SSS(\bfD \bfv))}}_2^2
    &\le c\, \norm{\smash{\bfF^*(\SSS(\Dhk\bfv_h))-\bfF^*(\SSS(\bfD \bfv))}}_2^2 
    \\&\quad + c\,  h^2\, \norm{\nabla \bfF(\bfD \bfv) }_2^2
    \\
    &\le c\, \bignorm{\smash{\bfF(\Dhk\bfv_h) - \bfF(\bfD\bfv)}}_2^2 
    + c\,  h^2\, \norm{\nabla \bfF(\bfD\bfv) }_2^2\,.
  \end{align*}
\end{proof}
The same method of proof of course also works for the $p$-Stokes
problem, i.e., we neglect the \hspace{-0.2mm}convective \hspace{-0.2mm}term \hspace{-0.2mm}in
\hspace{-0.2mm}\eqref{eq:p-navier-stokes}, \hspace{-0.2mm}Problem \hspace{-0.2mm}(P\hspace{-0.2mm}), \hspace{-0.2mm}Problem \hspace{-0.2mm}(P$_h$\hspace{-0.2mm}), \hspace{-0.2mm}Problem \hspace{-0.2mm}(\hspace{-0.2mm}Q),
\hspace{-0.2mm}and~\hspace{-0.2mm}\mbox{Problem~\hspace{-0.2mm}(\hspace{-0.2mm}Q$_h$\hspace{-0.2mm})}. Note that the dependence on $\delta^{-1}>0$ comes
solely from the convective term. Thus, we obtain for the $p$-Stokes problem a better dependence on the constants. 
\begin{theorem}
  \label{thm:error_stokes}
  Let $\SSS$ satisfy Assumption~\ref{assum:extra_stress} with
  $p\in(2,\infty)$ and $\delta\ge 0$,~let~$k\in \mathbb{N}$, and let
  $\bfg\in L^{p'}(\Omega)$. Moreover, let
  $(\bfv,q)^\top \in \Vo(0)\times \Qo$ be a solution of
  Problem~(Q) without the convective term (cf.~\eqref{eq:q1}, \eqref{eq:q2}) with
  $\bfF(\bfD \bfv) \in W^{1,2}(\Omega)$, %~and~${q\in  W^{1,p'}(\Omega)}$,
  and let $(\bfv_h,q_h)^\top \in {\Vhk(0)\times \Qhkco}$ be a solution
  of Problem (Q$_h$) without the terms coming from the convective term (cf.~\eqref{eq:primal1}) for $\alpha>0$. 
  Then, it holds 
  \begin{align*}
    \bignorm{\smash{\bfF\big(\Dhk \bfv_h\big)}\hspace{-0.25em}-\hspace{-0.2em} \bfF(\bfD
      \bfv)}_2^2 \!+\!
      m_{\phi_{\smash{\sssl}},h } (\bfv_h \hspace{-0.25em}-\hspace{-0.2em}\bfv)\leq c\, h^2 \norm{\nabla
    \bfF(\bfD \bfv) }_{2}^2\!+\!c\,\rho_{(\phi_{\abs{\bfD \bfv}})^*,\Omega}(h\nabla q)\hspace*{-1mm}
  \end{align*}
  with a~constant~$c>0$ depending only on the characteristics of
  $\SSS$, the chunkiness $\omega_0>0$, 
  $\alpha^{-1}>0$,~and~${k\in \mathbb{N}}$.
\end{theorem}

\begin{corollary}\label{cor:error_stokes}
  Let the assumptions of Theorem \ref{thm:error_stokes} be satisfied. Then,~it~holds
  \begin{align*}
    \bignorm{\smash{\bfF\big(\Dhk \bfv_h\big) \hspace{-0.06em}-\hspace{-0.06em} \bfF(\bfD
      \bfv)}}_2^2 \hspace{-0.06em}+\hspace{-0.05em} %\alpha
      \,m_{\phi_{\smash{\sssl}},h } (\bfv_h\hspace{-0.06em}-\hspace{-0.06em}\bfv)\hspace{-0.06em}\le\hspace{-0.06em} c\, h^2 \norm{\nabla\bfF(\bfD \bfv) }_2^2\hspace{-0.06em}+\hspace{-0.06em}c\,h^{p'}\!\rho_{\phi^*,\Omega}(\nabla q)
  \end{align*}
  with a constant $c>0$ depending only on the characteristics of
  $\SSS$, the~chunkiness~${\omega_0>0}$, 
  $\alpha^{-1}>0$, and  $k\in \mathbb{N}$.  {If, in addition, $\bfg\in L^{2}(\Omega)$,~then
  \begin{align*}
    &\bignorm{\smash{\bfF\big(\Dhk \bfv_h\big) - \bfF(\bfD
      \bfv)}}_2^2 +m_{\phi_{\smash{\sssl}},h } (\bfv_h-\bfv)
    \\[1.5mm]
    &\leq  c\, h^2 \smash{\big ( \norm{\bfF(\bfD \bfv)
      }_{1,2}^2+
      \big\|(\delta +\abs{\bfD\bfv})^{\smash{\frac{2-p}{2}}}\nabla q\big\|_2^2
      %\int_\Omega(\delta +\abs{\bfD\bfv})^{2-p}\abs{\nabla q}^2\, \textrm{d}x
      \big )}
  \end{align*}}
  with a constant $c>0$ depending only on the characteristics of
  $\SSS$, the~chunkiness~${\omega_0>0}$, 
  $\alpha^{-1}>0$, and $k\in \mathbb{N}$.
\end{corollary}

% \begin{corollary}\label{cor:error_stokes1}
%   Let $\SSS$ satisfy Assumption~\ref{assum:extra_stress} with
%   $p\in(2,\infty)$ and $\delta> 0$,~let ${k\in \mathbb{N}}$, and let
%   $\bfg\in L^{2}(\Omega)$. Moreover, let
%   $(\bfv,q)^\top \in \Vo(0)\times \Qo$ be a weak solution of
%   Problem~(Q) (cf.~\eqref{eq:q1}, \eqref{eq:q2}) without convective
%   term
%    %$\bfF(\bfD \bfv) \in W^{1,2}(\Omega)$ %and $q\in  W^{1,p'}(\Omega)$ %\footnote{Note that $W^{1,\phi^*}(\Omega)=W^{1,p'}(\Omega)$.},
%    and let $(\bfv_h,q_h)^\top \in {\Vhk(0)\times \Qhkco}$ be a solution
%    of Problem (Q$_h$) (cf.~\eqref{eq:primal1}) without convective term for $\alpha>0$.  Then,~it~holds  
%    \begin{align*}
%      &\bignorm{\smash{\bfF\big(\Dhk \bfv_h\big) - \bfF(\bfD
%        \bfv)}}_2^2 +m_{\phi_{\smash{\sssl}},h } (\bfv_h-\bfv)
%      \\
%      &\leq  c\, h^2 \Big ( \norm{\bfF(\bfD \bfv)
%        }_{1,2}^2+\int_\Omega(\delta
%        +\abs{\bfD\bfv})^{2-p}\abs{\nabla q}^2\, \textrm{d}x\Big )
%    \end{align*}
%    with a constant $c>0$ depending only on the characteristics of
%    $\SSS$,~the~chunkiness~${\omega_0>0}$, 
%    $\alpha^{-1}>0$, and  $k\in \mathbb{N}$.
% \end{corollary}

%\marginpar{add comment about improvement of FEM}%
\begin{remark}
  {\rm  {Note that Corollary \ref{cor:error_stokes} is the first result
     for $p$-Stokes problems with $p>2$ which shows a linear convergence rate
     under verifiable assumptions on the regularity of the
     solution. Moreover, the convergence
     rates are optimal for the different regularity assumptions (cf.~Section~\ref{sec:experiments}).}
    }
\end{remark}
% \begin{corollary}\label{cor:error_F*_stokes}
%  Let the assumptions of Theorem \ref{thm:error_stokes} be satisfied.~Then,~it~holds
%   \begin{align*}
%       \norm{\smash{\bfF^*(\PiDG\SSS(\Dhk\bfv_h))-\bfF^*(\SSS(\bfD \bfv))}}_2^2 \le c\,\norm{\smash{\bfF(\Dhk\bfv_h) - \bfF(\bfD \bfv)}}_2^2 +c\,  h^2\, \norm{\nabla \bfF(\bfD \bfv) }_2^2
%   \end{align*}
%   with a constant $c>0$ depending only on the characteristics of
%   $\SSS$, the~chunkiness~${\omega_0>0}$,   and  $k\in \mathbb{N}$.
% \end{corollary}

\begin{proof}[Proof of Theorem \ref{thm:error_stokes}]
  We proceed analogously to the proof of Theorem \ref{thm:error}. We start with \eqref{thm:error.1}, which reads
    \begin{align}\label{thm:error_stokes.1}
        \begin{aligned}
              c\,\bignorm{\smash{\bfF\big(\Dhk \bfv_h\big) - \bfF(\bfD
                \bfv)}}_2^2+c\,\alpha\,m_{\phi_{\smash{\sssl}},h } (\bfe_h)\leq  I_1+\alpha I_2 + I_3 + \alpha I_4\,,
        \end{aligned}
  \end{align}
   where  $I_i$, $i=1,\dots,4$, are defined in \eqref{thm:error.1}. Differing from the proof of Theorem \ref{thm:error}, due to the absence of the convective~term,~we~have~that
   \begin{align}\label{thm:error_stokes.2}
         I_1+\alpha I_2\leq J_1+J_3+J_5\,,
   \end{align}
   where $J_i$, $i=1,3,5$, are defined in \eqref{thm:error.4}. Then, resorting~in~\eqref{thm:error_stokes.1} to \eqref{thm:error_stokes.2}
    \eqref{thm:error.2}, \eqref{thm:error.3}, \eqref{thm:error.5a}, \eqref{eq:e9a}, and \eqref{thm:error.15}, we conclude that
    \begin{align}\label{thm:error_stokes.3}
      &c\,\bignorm{\smash{\bfF\big(\Dhk \bfv_h\big) - \bfF(\bfD
        \bfv)}}_2^2 + c\,\alpha\,m_{\phi_{\smash{\sssl}},h } (\bfv_h-\bfv)
      \\
      &\le  \big(\varepsilon+(1+\alpha)\,\kappa\, c_\varepsilon\big)\,c\, \bignorm{\smash{\bfF(\bfD\bfv)
        -\bfF\big(\Dhk\bfv_h\big)}}_2^2  +\varepsilon\,  c \,(1+\alpha)\,m_{\phi_{\smash{\sssl}},h } (\bfv_h-\bfv)\notag
      \\
      &\quad + \big(c_\varepsilon \,(1+c_\kappa+\alpha\,(c_\kappa+\kappa))\big)\,h^2\, \norm{\nabla 
        \bfF(\bfD \bfv) }_{2}^2+c_\varepsilon\,c_\kappa\, \rho_{(\phi_{\smash{\abs{\bfD\bfv}}})^*,\Omega}(h\,\nabla q)\,,\notag 
    \end{align}
    Choosing, for given $\alpha>0$, first $\varepsilon>0$ sufficiently
    small, and, then, for this fixed $\varepsilon>0$, 
    $\kappa>0$ sufficiently small, we can absorb the first two terms on the right-hand side of \eqref{thm:error_stokes.3} in the left-hand side and conclude 
    the existence of a
    constant $c>0$, depending only on the characteristics of $\SSS$,
    the chunkiness $\omega_0\hspace{-0.1em}>\hspace{-0.1em}0$, ${\alpha^{-1}\hspace{-0.1em}>\hspace{-0.1em}0}$, and $k\in \mathbb{N}$, such that
  \begin{align*}
    \bignorm{\smash{\bfF\big(\Dhk \bfv_h\big)}\hspace{-0.25em}-\hspace{-0.2em} \bfF(\bfD
      \bfv)}_2^2 \!+\!
      m_{\phi_{\smash{\sssl}},h } (\bfv_h \hspace{-0.25em}-\hspace{-0.2em}\bfv)\leq c\, h^2 \norm{\nabla
    \bfF(\bfD \bfv) }_{2}^2\!+\!c\,\rho_{(\phi_{\abs{\bfD \bfv}})^*,\Omega}(h\nabla q).\hspace*{-1mm}
  \end{align*}
    This completes the proof of Theorem \ref{thm:error_stokes}.
\end{proof}

\begin{proof}[Proof of Corollary \ref{cor:error_stokes}]
    We follow the arguments in the proof of Corollary \ref{cor:error} but now resort to Theorem \ref{thm:error_stokes}.
\end{proof}
%\begin{proof}[Proof of Corollary \ref{cor:error_stokes1}]
%    In \cite{br-reg} it is shown that $\bfg \in L^2(\Omega)$ yields
%    $\bfF(\bfD\bfv) \in W^{1,2}(\Omega)$. After that we follow the
%    arguments in the second part of the proof of Corollary \ref{cor:error} with the
%    simplifications due to the absence of the convective term.
%\end{proof}
% \begin{proof}[Proof of Corollary \ref{cor:error_F*_stokes}]
%     We follow the arguments in the proof of Corollary \ref{cor:error_F*} but now resort to Theorem \ref{thm:error_stokes}.
% \end{proof}
\section{Numerical experiments}\label{sec:experiments}

In this section, we apply the LDG scheme \eqref{eq:DG} (or \eqref{eq:primal1} and  \eqref{eq:primal2}) to solve
numerically the system~\eqref{eq:p-navier-stokes}~with $\SSS\colon\mathbb{R}^{d\times d}\to\mathbb{R}^{d\times d}$, for every $\bfA\in\mathbb{R}^{d\times d}$ defined via  ${\SSS(\bfA) \coloneqq (\delta+\vert \bfA^{\textup{sym}}\vert)^{p-2}\bfA^{\textup{sym}}}$,
where $\delta\coloneqq 1\textrm{e}{-}4$ and $ p>2$.
We approximate the discrete solution~${\bfv_h\in V^k_h}$ of the non-linear problem~\eqref{eq:DG}~deploying the Newton solver from \mbox{\textsf{PETSc}} (version 3.17.3), cf.~\cite{LW10}, with an absolute tolerance~of $\tau_{abs}\!=\! 1\textrm{e}{-}8$ and a relative tolerance of $\tau_{rel}\!=\!1\textrm{e}{-}10$. The linear system emerging in each Newton step is solved using a sparse direct solver from \textsf{MUMPS} (version~5.5.0),~cf.~\cite{mumps}. For the numerical flux \eqref{def:flux-S}, we choose the fixed parameter $\alpha=2.5$. This choice is in accordance with the choice in \mbox{\cite[Table~1]{dkrt-ldg}}. In the implementation, the uniqueness of the pressure is enforced via a zero mean condition.

All experiments were carried out using the finite element software package~\mbox{\textsf{FEniCS}} (version 2019.1.0), cf.~\cite{LW10}. 

For \hspace{-0.1mm}our \hspace{-0.1mm}numerical \hspace{-0.1mm}experiments, \hspace{-0.1mm}we \hspace{-0.1mm}choose \hspace{-0.1mm}$\Omega\!=\! (-1,1)^2$ \hspace{-0.1mm}and \hspace{-0.1mm}linear~\hspace{-0.1mm}elements,~\hspace{-0.1mm}i.e.,~\hspace{-0.1mm}${k\!=\! 1}$. We choose $\smash{\bfg\in L^{p'}(\Omega)}$ and boundary data $\bfv_0\in W^{\smash{1,1-\frac{1}{p}}}(\partial\Omega)$\footnote{The exact solution is not zero on the boundary of the computational domain. However, the error is clearly concentrated around the singularity and, thus, this small inconsistency with the setup of the theory does not have any influence on the results of this paper. In particular, note that Part I of the paper (cf.~\cite{kr-pnse-ldg-1})
already established at least the weak convergence of the method also for the fully non-homogeneous case.\vspace{-1.75cm}}  such that $\bfv\in W^{1,p}(\Omega)$ and $q \in \Qo$, for every $x\coloneqq (x_1,x_2)^\top\in \Omega$ defined by
\begin{align}
	\bfv(x)\coloneqq \vert x\vert^{\beta} (x_2,-x_1)^\top\,, \qquad q(x)\vcentcolon =25\,(\vert x\vert^{\gamma}-\langle\,\vert \!\cdot\!\vert^{\gamma}\,\rangle_\Omega)
\end{align}
are a solutions of  \eqref{eq:p-navier-stokes}. Here, we choose $\beta=1\textrm{e}{-}2$, which implies ${\bfF(\bfD\bfv)\in W^{1,2}(\Omega)}$. {Concerning the pressure regularity, we consider two cases: Namely,  
we choose either $\gamma= 1-\smash{\frac{2}{p'}}+1\textrm{e}{-}4$, which  just yields ${q \in
  W^{1,p'}(\Omega)}$ (case 1), or we choose ${\gamma= \beta\frac{p-2}{2}+1\textrm{e}{-}4}$, which just yields $(\delta+\abs{\bfD\bfv})^{\frac{2-p}{2}}\nabla q \in
  L^2(\Omega)$ (case 2). Thus,~for $\gamma= 1-\smash{\frac{2}{p'}}+1\textrm{e}{-}4$~(case~1), we can expect the
convergence~rate~$\smash{\frac{p'}{2}}$, while for $\gamma\hspace{-0.1em}= \hspace{-0.1em}\beta\frac{p-2}{2}+1\textrm{e}{-}4$ (case 2), we can expect~the~convergence~rate~$1$.~(cf.~Corollary~\ref{cor:error}).}%\newpage
 
We \hspace{-0.1mm}construct \hspace{-0.1mm}a \hspace{-0.1mm}initial \hspace{-0.1mm}triangulation \hspace{-0.1mm}$\mathcal
T_{h_0}$, \hspace{-0.1mm}where \hspace{-0.1mm}$h_0\hspace{-0.2em}=\hspace{-0.2em}\smash{\frac{1}{\sqrt{2}}}$, \hspace{-0.1mm}by \hspace{-0.1mm}subdividing~\hspace{-0.1mm}a~\hspace{-0.1mm}\mbox{rectangular} cartesian grid~into regular triangles with different orientations.  Finer triangulations~$\mathcal T_{h_i}$, $i=1,\dots,5$, where $h_{i+1}=\frac{h_i}{2}$ for all $i=1,\dots,5$, are 
obtained by
regular subdivision of the previous grid: Each \mbox{triangle} is subdivided
into four equal triangles by connecting the midpoints of the edges, i.e., applying the red-refinement rule, cf. \cite[Definition~4.8~(i)]{Ba16}.

Then, for the resulting series of triangulations $\mathcal T_{h_i}$, $i\!=\!1,\dots,5$, we apply~the~above Newton scheme to compute the corresponding numerical solutions $(\bfv_i,\bfL_i,\bfS_i)^\top\coloneqq \smash{(\bfv_{h_i},\bfL_{h_i},\bfS_{h_i})^\top\in V_{h_i}^k\times X_{h_i}^k\times X_{h_i}^k}$, $i=1,\dots,5$, 
and the error quantities
\begin{align*}
	\left.\begin{aligned}
		e_{\bfL,i}&\coloneqq \|\bfF(\bfL_i^{\textup{sym}})-\bfF(\bfL^{\textup{sym}})\|_2\,,\\%[0.75mm]
		e_{\bfS,i}&\coloneqq \|\bfF^*(\bfS_i)-\bfF^*(\bfS)\|_2\,,\\%[0.75mm]
		e_{\jump{},i}&\coloneqq m_{\phi_{\smash{\avg{\abs{\Pi_{h_i}^0\!\bfL_i^{\textup{sym}}}}}},h_i}(\bfv_i-\bfv)^{\smash{\frac{1}{2}}}\,,\\%[0.75mm]
		%e_{q,i}&\coloneqq (\|q_i-q\|_{p'}^{\smash{p'}})^{\smash{\frac{1}{2}}}\,,
	\end{aligned}\quad\right\}\quad i=1,\dots,5\,.
\end{align*}
As estimation of the convergence rates,  the experimental order of convergence~(EOC)
\begin{align*}
	\texttt{EOC}_i(e_i)\coloneqq \frac{\log(e_i/e_{i-1})}{\log(h_i/h_{i-1})}\,, \quad i=1,\dots,5\,,
\end{align*}
where for any $i= 1,\dots,5$, we denote by $e_i$
either 
$e_{\bfL,i}$ , $e_{\bfS,i}$, or
$e_{\jump{},i}$,~\mbox{resp.},~is~recorded.  For different values of $p\in \{2.25, 2.5, 2.75, 3, 3.25, 3.5\}$ and a
series of triangulations~$\mathcal{T}_{h_i}$, $i = 1,\dots,5$,
obtained by regular, global refinement as described above, the EOC is
computed and presented in Table~\ref{tab1}, Table~\ref{tab2}, and
Table~\ref{tab3}, % and Table~\ref{tab3}, 
respectively. {For both  case~1 and case 2 %the $p$-Navier--Stokes problem (cf.~\eqref{eq:p-navier-stokes}) and the $p$-Stokes problem~(cf.~\eqref{eq:p-navier-stokes}~without convective term), 
we observe the expected a convergence rate of about $\texttt{EOC}_i(e_i)\approx \smash{\frac{p'}{2}}$, and ${\texttt{EOC}_i(e_i)\approx 1}$, resp., $i=1,\dots, 5$, (cf.~Corollary
\ref{cor:error}).}

\begin{table}[H]
    \setlength\tabcolsep{1.9pt}
	\centering
	\begin{tabular}{c |c|c|c|c|c|c|c|c|c|c|c|c|} \cmidrule(){1-13}
    \multicolumn{1}{|c||}{\cellcolor{lightgray}$\gamma$}	
    & \multicolumn{6}{c||}{\cellcolor{lightgray} case 1 %$1-\tfrac{2}{p'}+1\textrm{e}{-}4$
    }   & \multicolumn{6}{c|}{\cellcolor{lightgray} case 2 %$\alpha\tfrac{p-2}{2}+1\textrm{e}{-}4$
    }\\ 
			\hline 
		   
		    \multicolumn{1}{|c||}{\cellcolor{lightgray}\diagbox[height=1.1\line,width=0.11\dimexpr\linewidth]{\vspace{-0.6mm}$i$}{\\[-5mm] $p$}}
		    & \cellcolor{lightgray}2.25 & \cellcolor{lightgray}2.5  & \cellcolor{lightgray}2.75  &  \cellcolor{lightgray}3.0 & \cellcolor{lightgray}3.25  & \multicolumn{1}{c||}{\cellcolor{lightgray}3.5} &  \multicolumn{1}{c|}{\cellcolor{lightgray}2.25}   & \cellcolor{lightgray}2.5  & \cellcolor{lightgray}2.75  & \cellcolor{lightgray}3.0  & \cellcolor{lightgray}3.25 &   \cellcolor{lightgray}3.5 \\ \hline\hline
			\multicolumn{1}{|c||}{\cellcolor{lightgray}$1$}                		& 0.888 & 0.819 & 0.771 & 0.734 & 0.706 & \multicolumn{1}{c||}{0.684} & \multicolumn{1}{c|}{1.360} & 1.109 & 1.000 & 0.940 & 0.898 & 0.870 \\ \hline
			\multicolumn{1}{|c||}{\cellcolor{lightgray}$2$}                  	& 0.898 & 0.830 & 0.781 & 0.745 & 0.716 & \multicolumn{1}{c||}{0.693} & \multicolumn{1}{c|}{1.116} & 1.065 & 1.032 & 1.002 & 0.983 & 0.969 \\ \hline
			\multicolumn{1}{|c||}{\cellcolor{lightgray}$3$}                     & 0.899 & 0.832 & 0.784 & 0.748 & 0.720 & \multicolumn{1}{c||}{0.697} & \multicolumn{1}{c|}{1.010} & 1.021 & 1.019 & 1.011 & 0.996 & 0.987 \\ \hline
			\multicolumn{1}{|c||}{\cellcolor{lightgray}$4$}               		& 0.900 & 0.833 & 0.785 & 0.749 & 0.721 & \multicolumn{1}{c||}{0.699} & \multicolumn{1}{c|}{0.982} & 1.006 & 1.008 & 1.008 & 1.003 & 0.994 \\ \hline
			\multicolumn{1}{|c||}{\cellcolor{lightgray}$5$}               		& 0.900 & 0.833 & 0.786 & 0.750 & 0.722 & \multicolumn{1}{c||}{0.700} & \multicolumn{1}{c|}{0.977} & 1.002 & 1.004 & 1.006 & 1.005 & 1.000 \\ \hline\hline
            \multicolumn{1}{|c||}{\cellcolor{lightgray}\small expected}         & 0.900 & 0.833 & 0.786 & 0.750 & 0.722 & \multicolumn{1}{c||}{0.700} & \multicolumn{1}{c|}{1.000} & 1.000 & 1.000 & 1.000 & 1.000 & 1.000 \\ \hline
	\end{tabular}\vspace{-2mm}
	\caption{\small Experimental order of convergence: $\texttt{EOC}_i(e_{\bfL,i})$,~${i=1,\dots,5}$.}
%\notag	
\label{tab1}
\end{table}\vspace{-8mm}

\begin{table}[H]
    \setlength\tabcolsep{1.9pt}
	\centering
	\begin{tabular}{c |c|c|c|c|c|c|c|c|c|c|c|c|} \cmidrule(){1-13}
	\multicolumn{1}{|c||}{\cellcolor{lightgray}$\gamma$}	
    & \multicolumn{6}{c||}{\cellcolor{lightgray}case 1}   & \multicolumn{6}{c|}{\cellcolor{lightgray}case 2}\\ 
			\hline 
		   
		    \multicolumn{1}{|c||}{\cellcolor{lightgray}\diagbox[height=1.1\line,width=0.11\dimexpr\linewidth]{\vspace{-0.6mm}$i$}{\\[-5mm] $p$}}
		    & \cellcolor{lightgray}2.25 & \cellcolor{lightgray}2.5  & \cellcolor{lightgray}2.75  &  \cellcolor{lightgray}3.0 & \cellcolor{lightgray}3.25  & \multicolumn{1}{c||}{\cellcolor{lightgray}3.5} &  \multicolumn{1}{c|}{\cellcolor{lightgray}2.25}   & \cellcolor{lightgray}2.5  & \cellcolor{lightgray}2.75  & \cellcolor{lightgray}3.0  & \cellcolor{lightgray}3.25 &   \cellcolor{lightgray}3.5 \\ \hline\hline
			\multicolumn{1}{|c||}{\cellcolor{lightgray}$1$}                		& 0.895 & 0.824 & 0.776 & 0.739 & 0.711 & \multicolumn{1}{c||}{0.688} & \multicolumn{1}{c|}{1.533} & 1.186 & 0.983 & 0.893 & 0.857 & 0.852 \\ \hline
			\multicolumn{1}{|c||}{\cellcolor{lightgray}$2$}                  	& 0.898 & 0.830 & 0.782 & 0.746 & 0.718 & \multicolumn{1}{c||}{0.696} & \multicolumn{1}{c|}{1.487} & 1.285 & 1.134 & 1.045 & 1.002 & 0.970 \\ \hline
			\multicolumn{1}{|c||}{\cellcolor{lightgray}$3$}                     & 0.900 & 0.832 & 0.784 & 0.749 & 0.721 & \multicolumn{1}{c||}{0.699} & \multicolumn{1}{c|}{1.245} & 1.161 & 1.137 & 1.085 & 1.048 & 1.028 \\ \hline
			\multicolumn{1}{|c||}{\cellcolor{lightgray}$4$}               		& 0.900 & 0.833 & 0.785 & 0.750 & 0.722 & \multicolumn{1}{c||}{0.700} & \multicolumn{1}{c|}{1.075} & 1.056 & 1.061 & 1.064 & 1.043 & 1.030 \\ \hline
			\multicolumn{1}{|c||}{\cellcolor{lightgray}$5$}               		& 0.900 & 0.833 & 0.786 & 0.750 & 0.722 & \multicolumn{1}{c||}{0.700} & \multicolumn{1}{c|}{1.011} & 1.017 & 1.021 & 1.028 & 1.031 & 1.024 \\ \hline\hline
			\multicolumn{1}{|c||}{\cellcolor{lightgray}\small expected}         & 0.900 & 0.833 & 0.786 & 0.750 & 0.722 & \multicolumn{1}{c||}{0.700} & \multicolumn{1}{c|}{1.000} & 1.000 & 1.000 & 1.000 & 1.000 & 1.000 \\ \hline
	\end{tabular}\vspace{-2mm}
	\caption{Experimental order of convergence: $\texttt{EOC}_i(e_{\jump{},i})$,~${i=1,\dots,5}$.}
	\label{tab2}
\end{table}\vspace{-8mm}

 \begin{table}[H]
     \setlength\tabcolsep{1.9pt}
 	\centering
 	\begin{tabular}{c |c|c|c|c|c|c|c|c|c|c|c|c|} \cmidrule(){1-13}
 	\multicolumn{1}{|c||}{\cellcolor{lightgray}$\gamma$}	
    & \multicolumn{6}{c||}{\cellcolor{lightgray}case 1}   & \multicolumn{6}{c|}{\cellcolor{lightgray}case 2}\\ 
 	\hline 
		   
		    \multicolumn{1}{|c||}{\cellcolor{lightgray}\diagbox[height=1.1\line,width=0.11\dimexpr\linewidth]{\vspace{-0.6mm}$i$}{\\[-5mm] $p$}}
		    & \cellcolor{lightgray}2.25 & \cellcolor{lightgray}2.5  & \cellcolor{lightgray}2.75  &  \cellcolor{lightgray}3.0 & \cellcolor{lightgray}3.25  & \multicolumn{1}{c||}{\cellcolor{lightgray}3.5} &  \multicolumn{1}{c|}{\cellcolor{lightgray}2.25}   & \cellcolor{lightgray}2.5  & \cellcolor{lightgray}2.75  & \cellcolor{lightgray}3.0  & \cellcolor{lightgray}3.25 &   \cellcolor{lightgray}3.5 \\ \hline\hline
			\multicolumn{1}{|c||}{\cellcolor{lightgray}$1$}                		& 0.888 & 0.819 & 0.771 & 0.734 & 0.706 & \multicolumn{1}{c||}{0.684} & \multicolumn{1}{c|}{1.360} & 1.109 & 1.000 & 0.940 & 0.898 & 0.870 \\ \hline
			\multicolumn{1}{|c||}{\cellcolor{lightgray}$2$}                  	& 0.898 & 0.830 & 0.781 & 0.744 & 0.716 & \multicolumn{1}{c||}{0.693} & \multicolumn{1}{c|}{1.116} & 1.066 & 1.032 & 1.001 & 0.983 & 0.968 \\ \hline
			\multicolumn{1}{|c||}{\cellcolor{lightgray}$3$}                     & 0.899 & 0.832 & 0.784 & 0.748 & 0.720 & \multicolumn{1}{c||}{0.697} & \multicolumn{1}{c|}{1.010} & 1.022 & 1.019 & 1.010 & 0.996 & 0.987 \\ \hline
			\multicolumn{1}{|c||}{\cellcolor{lightgray}$4$}               		& 0.900 & 0.833 & 0.785 & 0.749 & 0.721 & \multicolumn{1}{c||}{0.699} & \multicolumn{1}{c|}{0.982} & 1.006 & 1.008 & 1.009 & 1.003 & 0.994 \\ \hline
			\multicolumn{1}{|c||}{\cellcolor{lightgray}$5$}               		& 0.900 & 0.833 & 0.786 & 0.750 & 0.722 & \multicolumn{1}{c||}{0.700} & \multicolumn{1}{c|}{0.977} & 1.002 & 1.004 & 1.006 & 1.005 & 1.000 \\ \hline\hline
			\multicolumn{1}{|c||}{\cellcolor{lightgray}\small expected}         & 0.900 & 0.833 & 0.786 & 0.750 & 0.722 & \multicolumn{1}{c||}{0.700} & \multicolumn{1}{c|}{1.000} & 1.000 & 1.000 & 1.000 & 1.000 & 1.000 \\ \hline
	\end{tabular}\vspace{-2mm}
 	\caption{Experimental order of convergence: $\texttt{EOC}_i(e_{\bfS,i})$,~${i=1,\dots,5}$.}
 	\label{tab3}
 \end{table}\vspace{-1cm}

\def\cprime{$'$} \def\cprime{$'$} \def\cprime{$'$}

\end{document}